\documentclass[a4paper, 12pt]{amsart}

\usepackage{amssymb}
\usepackage{amsthm}
\usepackage{amsmath}
\usepackage[mathscr]{eucal}
\usepackage{comment}

\theoremstyle{plain}
\newtheorem{thm}{Theorem}[section]		
\newtheorem{prop}[thm]{Proposition}
\newtheorem{cor}[thm]{Corollary}
\newtheorem{lem}[thm]{Lemma}
\newtheorem{clm}[thm]{Claim}

\theoremstyle{definition}
\newtheorem{df}{Definition}[section]
\newtheorem{exam}{Example}[section]

\theoremstyle{remark}
\newtheorem{rmk}{Remark}[section]
\newtheorem*{ac}{Acknowledgements}

\newcommand{\nn}{\mathbb{N}}

\newcommand{\rr}{\mathbb{R}}

\newcommand{\mgn}{\infty}

\DeclareMathOperator{\card}{card}

\DeclareMathOperator{\di}{diam}

\DeclareMathOperator{\cl}{CL}

\makeatletter
\@addtoreset{equation}{section}

\makeatother

\begin{document}

\title[Quasi-symmetric invariant properties]
{Quasi-symmetric invariant properties of Cantor metric spaces}
\author[Yoshito Ishiki]
{Yoshito Ishiki}
\address[Yoshito Ishiki]
{\endgraf
Graduate School of Pure and Applied Sciences
\endgraf
University of Tsukuba
\endgraf
Tennodai 1-1-1, Tsukuba, Ibaraki, 305-8571, Japan}
\email{ishiki@math.tsukuba.ac.jp}

\date{\today}

\begin{abstract}
For metric spaces, the doubling property, the uniform disconnectedness, and the uniform perfectness are known as quasi-symmetric invariant properties. 
The David-Semmes uniformization theorem states that 
if a compact metric space satisfies all the three properties, then it is quasi-symmetrically equivalent to the middle-third Cantor set. 
We say that a Cantor metric  space is standard if it satisfies all the three properties; otherwise, it is exotic. 
In this paper, 
we conclude that for each of  exotic types the class of all the conformal gauges of Cantor metric spaces exactly has continuum cardinality. 
As a byproduct of our study, 
we state that there exists a Cantor metric space with prescribed Hausdorff dimension and Assouad dimension.  
\end{abstract}
\maketitle

\section{Introduction}
The concept of quasi-symmetric maps between metric spaces provides us various applications, 
especially from a viewpoint of geometric analysis of metric measure spaces (see e.g., \cite{H,S}), 
or a viewpoint of the conformal dimension theory (see e.g., \cite{MT}). 
For a homeomorphism $\eta:[0,\mgn)\to [0.\mgn)$,
a homeomorphism $f:X\to Y$ between metric spaces is said to be \emph{$\eta$-quasi-symmetric} 
if 
\[
\frac{d_Y(f(x),f(y))}{d_Y(f(x),f(z))}\le \eta\left(\frac{d_X(x,y)}{d_X(x,z)}\right)
\] holds for all distinct $x,y,z\in X$, 
where $d_X$ is the metric on $X$ and $d_Y$ the metric on $Y$.
A homeomorphism $f:X\to Y$ is \emph{quasi-symmetric} if 
it is $\eta$-quasi-symmetric for some $\eta$. 
The composition of any two quasi-symmetric maps is quasi-symmetric. 
The inverse of any quasi-symmetric map is also quasi-symmetric. 
The quasi-symmetry gives us an equivalent relation between metric spaces. \par
In this paper, we focus on the following quasi-symmetric invariant properties of metric spaces: 
the doubling property, the uniform disconnectedness, and the uniform perfectness 
(see Section \ref{sec:pre} for the definitions). 
David and Semmes \cite{DS} have proven the so-called uniformization theorem which states that 
every uniformly disconnected, uniformly perfect, doubling compact metric space is 
quasi-symmetrically equivalent to the middle-third Cantor set (\cite[Proposition 15.11]{DS}). 
The David-Semmes uniformization theorem can be considered as a 
quasi-symmetric version of the well-known Brouwer characterization of Cantor spaces 
(\cite{B}, see e.g., \cite[Theorem 30.3]{W}), 
where a Cantor space means  a topological space homeomorphic to the middle-third Cantor set.
We study the three quasi-symmetric invariant properties of Cantor metric spaces. 
We attempt to complement the David-Semmes uniformization theorem.
\par
Before stating our results, for the sake of simplicity, 
we introduce the following notations:
\begin{df}
If a metric space $(X,d)$ with metric $d$ satisfies a property $P$, 
then we write $T_P(X,d)=1$; otherwise, $T_P(X,d)=0$.
For a triple $(u,v,w)\in \{0,1\}^3$, 
we say that a metric space $(X,d)$ \emph{has type $(u,v,w)$} 
if we have
\[
T_{D}(X,d)=u,\quad T_{UD}(X,d)=v,\quad T_{UP}(X,d)=w,
\] 
where $D$ means the doubling property, $UD$ the uniform disconnectedness, and $UP$ the uniform perfectness.
\end{df}
We say that  a Cantor metric space is \emph{standard} if it has type $(1,1,1)$;
otherwise, \emph{exotic}.
For example, the middle-third Cantor set is standard. 
We consider the problem on an abundance of 
the quasi-symmetric equivalent classes of exotic Cantor metric spaces. 
\par
For a metric space $(X,d)$, we denote by $\mathcal{G}(X,d)$ the \emph{conformal gauge} of $(X,d)$ defined as 
the quasi-symmetric equivalent class of $(X,d)$.
The conformal gauge of metric spaces is a basic concept in the conformal dimension theory 
(see e.g., \cite{MT}).
For each $(u,v,w)\in \{0,1\}^3$, we define 
\[
\mathscr{M}(u,v,w)=\{\,\mathcal{G}(X,d)\mid \text{$(X,d)$ is a Cantor space of type $(u,v,w)$}\,\}.
\]
\par
The  David-Semmes uniformization theorem mentioned above states that $\mathscr{M}(1,1,1)$ is a singleton. 
It is  intuitively expected that $\mathscr{M}(u,v,w)$ has infinite cardinality for each exotic type $(u,v,w)$. 
As far as the author knows, the caridinality of the class of the conformal gauges of Cantor metric spaces has not yet been studied. 
\par
As the main result of this paper, 
we conclude that the cardinality of  the class of all conformal gauges of exotic Cantor metric  spaces is equal to the continuum $2^{\aleph_0}$. More precisely, we prove the following:
\begin{thm}\label{thm:many}
For every $(u,v,w)\in \{0,1\}^3$ except $(1,1,1)$, we have
\[
\card(\mathscr{M}(u,v,w))=2^{\aleph_0}, 
\]
where the symbol $\card$ denotes the cardinality.
\end{thm}
In the proof of Theorem \ref{thm:many},  
the following quasi-symmetric invariant plays an important role. 
\begin{df}\label{def:sp}
For a property $P$ of metric spaces,  
and for a metric space $(X,d)$ we define $S_{P}(X,d)$ as the set of all points in $X$ of which no neighborhoods 
satisfy $P$. 
\end{df}
\begin{rmk}\label{rmk:qsinv}
If $P$ is a quasi-symmetric invariant property (e.g., $D$, $UD$ or $UP$),  then $S_P(X,d)$ is a quasi-symmetric invariant. Namely, 
if $(X,d_X)$ and $(Y,d_Y)$ are quasi-symmetrically equivalent, 
then so are $S_P(X,d_X)$ and $S_P(Y,d_Y)$.
\end{rmk}
To prove Theorem \ref{thm:many}, 
 we introduce the following notion: 
\begin{df}\label{def:spike}
For a property $P$ of metric spaces, 
we say that a metric space $(X,d)$ is \emph{a $P$-spike space} if $S_P(X,d)$ is a singleton. 
\end{df}
In order to guarantee the existence of $D$, $UD$ and $UP$-spike Cantor metric spaces,  
we develop a new operation of metric spaces, say the \emph{telescope spaces}. 
Our telescope space is constructed as a direct sum with contracting factors and the point at infinity determined as the convergent point of the contracting factors (see Section  \ref{sec:telescope}). 
\par 
The outline of the proof of Theorem \ref{thm:many} is as follows: 
We first construct a family $\{\Xi(x)\}_{x\in I}$ of continuum many closed sets in the middle-third Cantor set whose members are not homeomorphic to each other.  
By using appropriate $D$, $UD$ and $UP$-spike Cantor metric spaces,   
for each member $\Xi(x)$, for each exotic type  $(u,v,w)$ and for each failing property $P\in \{D,UD,UP\}$  of $(u,v,w)$, we can obtain a Cantor metric space $(X,d)$ of type $(u,v,w)$ satisfying
$S_P(X,d)=\Xi(x)$. 
Since $S_P$ is a quasi-symmetric invariant for $D$, $UD$, and $UP$, 
we obtain continuum many Cantor metric spaces in $\mathscr{M}(u,v,w)$. 
\par
As a natural question, 
we consider the problem whether a Cantor metric space $(X,d)$ with 
$S_P(X,d)=X$ exists, where $P$ means $D$, $UD$ or $UP$. 
\begin{df}
For a triple $(u,v,w)\in \{0,1\}^3$, 
we say that a metric space $(X,d)$ has  \emph{totally exotic type $(u,v,w)$} 
if $(X,d)$ has exotic type $(u,v,w)$, 
and if $S_{P}(X,d)=X$ holds for all $P\in \{D,UD,UP\}$ with $T_P(X,d)=0$. 
\end{df}
As the other result, 
we prove the existence of  totally exotic Cantor metric spaces for all the possible types. 
\begin{thm}\label{thm:totexo}
For every $(u,v,w)\in \{0,1\}^3$ except $(1,1,1)$, there exists a Cantor metric space of  totally exotic type $(u,v,w)$. 
\end{thm}
Theorem \ref{thm:totexo} states an abundance of examples of  exotic Cantor metric spaces in a different way from Theorem \ref{thm:many}. 
\par
To prove Theorem \ref{thm:totexo}, we introduce the notions of the \emph{sequentially metrized Cantor spaces} and the \emph{kaleidoscope spaces}.  
We first explain the sequentially metrized Cantor spaces. 
Let $2^{\nn}$ denote the set of all maps from $\nn$ to $\{0,1\}$. 
For each $u\in (0,1)$, 
the set $2^{\nn}$ equipped with an ultrametric $d$ defined by $d(x,y)=u^{\min\{n\in \nn \mid x_n\neq y_n\}}$ 
becomes a Cantor space. 
In the study of David-Semmes \cite{DS}, or in more preceding studies, 
the metric space $(2^{\nn},d)$ is often utilized as an abstract Cantor space rather than the middle-third one.
The point in the proceeding studies is to use a geometric sequence $\{u^n\}_{n\in \nn}$ 
in the definition of $d$. 
We modify such a familiar construction by using more general sequences, say \emph{shrinking sequences}, 
that are non-increasing and converging to $0$.
Our sequentially metrized Cantor space means the metric space $2^{\nn}$ equipped with a metric constructed by a shrinking sequence (see Section \ref{sec:smcs}). 
In the proof of Theorem \ref{thm:totexo},  Cantor metric spaces of totally exotic types $(1,1,0)$, $(0,1,1)$ and $(0,1,0)$ are obtained as sequentially metrized Cantor spaces for some suitable shrinking sequences. 
\par
We next explain the kaleidoscope spaces. 
Our kaleidoscope space is defined as the countable product of equally divided points in $[0,1]$ 
equipped with a supremum metric distorted by an increasing sequence (see Section \ref{sec:totexo}). 
In the proof of Theorem \ref{thm:totexo}, Cantor metric spaces of totally exotic types $(1,0,1)$, $(1,0,0)$ and $(0,0,0)$ are obtained by applying the construction of the kaleidoscope spaces. 
\par
As an application of our studies of Cantor metric spaces, we examine the prescribed Hausdorff and Assouad dimensions problem. 
For a metric space $(X,d)$, we denote by $\dim_H(X,d)$ the Hausdorff dimension of $(X,d)$, 
and by $\dim_A(X,d)$ the Assouad dimension.
In general, the Hausdorff dimension does not succeed the Assouad dimension (see Subsection \ref{Assouad} for the basics of Assouad dimension).
\begin{thm}\label{thm:ab}
For each pair $(a,b)\in [0,\infty]^2$ with $a\le b$,  
 there exists a Cantor metric space $(X,d)$ with 
\[
\dim_H(X,d)=a,\quad \dim_A(X,d)=b.
\] 
\end{thm}
Our constructions of Cantor metric spaces mentioned above enable us to prove Theorem \ref{thm:ab}.
\par
The organization of this paper is as follows: 
In Section \ref{sec:pre}, we explain the basic facts of metric spaces.
In Section \ref{sec:telescope}, we introduce the notion of the telescope spaces, 
and study their basic properties. 
In Section \ref{sec:spike}, we prove the existence of the $D$, $UD$ and $UP$-spike Cantor metric spaces. 
In Section \ref{sec:prf}, we prove Theorem \ref{thm:many}. 
In Section \ref{sec:smcs}, we discuss the basic properties of the sequentially metrized Cantor spaces. 
In Section \ref{sec:totexo}, we introduce the notion of the kaleidoscope spaces, 
and prove Theorem \ref{thm:totexo}. 
In Section \ref{sec:prescribe}, we prove Theorem \ref{thm:ab}. 
\begin{ac}
The author would like to thank Professor Koichi Nagano  for his advice and constant encouragement.
\end{ac}

\section{Preliminaries}\label{sec:pre}
\subsection{Metric Spaces}
Let $(X,d)$ be a metric space.
For a point $x\in X$ and for a positive number $r\in (0,\mgn)$, 
we denote by $U(x,r)$ the open metric ball with 
center $x$ and radius $r$, and by $B(x,r)$ the closed one.
For a subset $A$ of $X$, 
we denote by $\di(A)$ the diameter of $A$.\par
For $\delta \in (0,\mgn)$, 
we denote by $\mathcal{F}_{\delta}(X)$ the set of all subsets of $X$ with diameter smaller than $\delta$. 
For a non-negative number $s\in [0,\mgn)$, 
we denote by $\mathcal{H}^s$ 
the $s$--dimensional Hausdorff measure on $X$ defined as 
$\mathcal{H}^s(A)=\sup_{\delta\in (0,\infty)}\mathcal{H}_{\delta}^s(A)$, where 
\[
\mathcal{H}_{\delta}^s(A)=\inf\left\{\, \sum_{i=1}^{\mgn}\di(A_i)^s\ \middle| \ A\subset\bigcup_{i=1}^{\mgn}A_i,\ A_i\in \mathcal{F}_{\delta}(X)\,\right\}.
\]
For a subset $A$ of $X$, we denote by $\dim_H(A)$ the Hausdorff dimension of $A$ defined as
\begin{align*}
\dim_H(A)&=\sup\{\,s\in[0,\mgn)\mid \mathcal{H}^s(A)=\mgn\,\}\\
             &=\inf\{\,s\in [0,\mgn)\mid \mathcal{H}^s(A)=0\,\}.
\end{align*}\par
Let $(X,d_X)$ and $(Y,d_Y)$ be metric spaces. 
For $c\in (0,\infty)$, 
a map $f:X\to Y$ is said to be \emph{$c$-Lipschitz} 
if for all $x,y \in X$ we have $d_Y(f(x),f(y))\le cd_X(x,y)$. 
A map between metric spaces is \emph{Lipschitz} if it is $c$-Lipschitz for some $c$.
A map $f:X\to Y$ is said to be \emph{$c$-bi-Lipschitz} if for all $x,y\in X$ we have  
\[
c^{-1}d_X(x,y)\le d_Y(f(x),f(y))\le cd_X(x,y).
\] 
A map between metric spaces is \emph{bi-Lipschitz} if it is $c$-bi-Lipschitz for some $c$. 
Two metric spaces are said to be \emph{bi-Lipschitz equivalent} 
if there exists a bi-Lipschitz homeomorphism between them. 
Note that every bi-Lipschitz map is quasi-symmetric. 
\subsection{Cantor Metric Spaces}
A topological space is said to be \emph{$0$-dimensional} 
if it admits a clopen base.
A metric space $(X,d)$ is called an \emph{ultrametric space} if 
for all $x,y,z\in X$ we have the so-called ultrametric triangle inequality 
\[
d(x,y)\le \max\{d(x,z),d(z,y)\};
\]
in this case, $d$ is called an \emph{ultrametric}.
Every ultrametric space is $0$-dimensional.\par
We recall the following characterization  of Cantor spaces due to Brouwer (\cite{B}, 
see e.g.,  \cite[Theorem 30.3]{W}):
\begin{thm}[\cite{B}]\label{Brouwer}
Every $0$-dimensional, compact metric space possessing  no isolated point is a Cantor space.
\end{thm}\par
The following example can  be seen in \cite{DS}:
\begin{exam}\label{exm:sym}
Let $2^{\nn}$ denote the set of all maps from $\nn$ to $\{0,1\}$. 
Let $e$ be a metric on $2^{\nn}$ defined by 
\[
e(x,y)=3^{-\min\{n\in \nn\mid x_n\neq y_n\}}.
\]
The metric $e$ is an ultrametric on $2^{\nn}$. 
By the Brouwer theorem \ref{Brouwer}, the metric space $(2^{\nn},e)$ is a Cantor space.  
\end{exam}
\subsection{Doubling Property} 
For a positive integer $N\in \nn$, 
a metric space $(X,d)$ is said to be \emph{$N$-doubling} 
if every closed metric ball with radius $r$ can be covered by at most $N$ closed metric balls with radius $r/2$.
A metric space is \emph{doubling} if it is $N$-doubling for some $N$.\par
The doubling property is hereditary. 
Namely, every subspace of an $N$-doubling metric space is $N$-doubling. 
\begin{exam}
The middle-third Cantor set $(\Gamma, d_{\Gamma})$ is doubling 
since the real line is doubling.
\end{exam}
Let $(X,d)$ be a metric space, and 
let $A$ be a subset of $X$. 
For  $r\in (0,\infty)$, 
a subset $S$ of $A$ is said to be $r$-separated in $A$ 
if for all distinct points $x,y\in S$ we have $d(x,y)\ge r$. 
\begin{lem}\label{lem:separated}
A metric space $(X,d)$ is doubling if and only if 
there exists $M\in \nn$ such that 
for each $r\in (0,\infty)$ and for each $x\in X$, 
the cardinality of an arbitrary $(r/2)$-separated set in $B(x,r)$ is at most $M$.
\end{lem}

\subsection{Uniform Disconnectedness}
Let $(X,d)$ be a metric space. 
For  $\delta\in (0,1)$, 
a finite sequence $x:\{0,1,\dots, N\}\to X$ is said to be a \emph{$\delta$-chain in $(X,d)$} 
if $d(x(i-1),x(i))\le \delta d(x(0),x(N))$ for all $i\in \{1,\dots, N\}$. 
A $\delta$-chain in $(X,d)$ is called \emph{trivial} if it is constant. 
\par
For $\delta\in (0,1)$, 
a metric space $(X,d)$ is said to be \emph{$\delta$-uniformly disconnected} 
if every $\delta$-chain in $(X,d)$ is trivial. 
A metric space is \emph{uniformly disconnected} 
if it is $\delta$-uniformly disconnected for some $\delta$.\par
The uniformly disconnectedness is hereditary. 
Namely, every subspace of a $\delta$-uniformly disconnected metric space is $\delta$-uniformly disconnected. 
\par
By the definition of the uniform disconnectedness and the ultrametric triangle inequality, we see the following:
\begin{prop}\label{prop:ultraud}
Let $(X,d)$ be an ultrametric space. Then for every $\delta\in (0,1)$ the space $(X,d)$ is $\delta$-uniformly disconnected. 
\end{prop}
We have already known the following characterization of the uniform disconnectedness 
(see e.g., \cite{DS}, \cite{MT}):

\begin{prop}\label{prop:ult}
A metric space is uniformly disconnected if and only if it is bi-Lipschitz equivalent to an ultrametric space.
\end{prop}
\begin{rmk}\label{rmk:ud}
More precisely, every $\delta$-uniformly disconnected metric space is $\delta^{-1}$-bi-Lipschitz equivalent to an  ultrametric space. 
\end{rmk}

\begin{exam}\label{exm:t}
The middle-third Cantor set $(\Gamma,d_{\Gamma})$ is uniformly disconnected. 
This claim can be verified as follows:
Take the Cantor space $(2^{\nn},e)$ mentioned in Example \ref{exm:sym}. 
The ternary corresponding map $T:2^{\nn}\to \Gamma$ defined as 
$T(x)=\sum_{i=1}^{\mgn}(2/3^i)x_i$ is a bi-Lipschitz homeomorphism.
Since $(2^{\nn},e)$ is an ultrametric space, 
Proposition \ref{prop:ult} tells us that $(\Gamma,d_{\Gamma})$ is uniformly disconnected.
\end{exam}

\begin{rmk}\label{rmk:c}
More precisely, we see the following:
\begin{enumerate}
\item For all $x,y\in2^{\nn}$, we have
\[
\frac{3}{2}e(x,y)\le d_{\Gamma}(T(x),T(y))\le \frac{5}{2}e(x,y).
\]
Indeed, if we put $n=\min\{k\in \nn\mid x_k\neq y_k\}$, then 
\begin{align*}
d_{\Gamma}(T(x),T(y))&=\left|\sum_{i=1}^{\infty}\frac{2x_i}{3^i}-\sum_{i=1}^{\infty}\frac{2y_i}{3^i}\right|\le 
\frac{2}{3^n}+\sum_{i=n+1}^{\infty} \frac{2}{3^i}=\frac{5}{2}e(x,y),\\
d_{\Gamma}(T(x),T(y))&=\left|\sum_{i=1}^{\infty}\frac{2x_i}{3^i}-\sum_{i=1}^{\infty}\frac{2y_i}{3^i}\right|
\ge \frac{2}{3^n}-\sum_{i=n+1}^{\infty} \frac{2}{3^i}=\frac{3}{2}e(x,y).
\end{align*}
\item For every $\delta\in (0,3/5)$, the space $(\Gamma,d_{\Gamma})$ is $\delta$-uniformly disconnected. 
Indeed, for each $\delta$-chain $x$ in $(\Gamma,d_{\Gamma})$,
 the sequence $T\circ x$ is a $(5\delta/3)$-chain in $(2^{\nn},e)$.
 \end{enumerate}
\end{rmk}

\subsection{Uniform Perfectness}
For $\rho\in (0,1]$, 
a metric space $(X,d)$ is said to be \emph{$\rho$-uniformly perfect} if 
for every $x\in X$, and for every $r\in (0,\di(X))$, 
the set $B(x,r)\setminus U(x,\rho r)$ is non-empty.
A metric space is \emph{uniformly perfect} if 
it is $\rho$-uniformly perfect for some $\rho$.
\par
From the definition we derive the following:
\begin{lem}\label{lem:ext}
Let $(X,d)$ be a $\rho$-uniformly perfect bounded metric space.
For $\lambda\in (1,\mgn)$, put $\mu=\rho/(2\lambda)$. 
Then for every $x\in X$ and for every $r\in (0,\lambda \di (X))$, 
the set $B(x,r)\setminus U(x,\mu r)$ is non-empty, 
and $B(x,\mu r)$ is a proper subset of $X$.
\end{lem}
\begin{proof}
Assume first that $B(x,r)$ is a proper subset of $X$. This implies $r<\di(X)$. Since $(X,d)$ is $\rho$-uniformly perfect, 
it is also $\mu$-uniformly perfect. Hence $B(x,r)\setminus U(x,\mu r)$ is non-empty. 
Assume second that $B(x,r)= X$. By the definition of $\mu$, 
we have $\di(B(x,\mu r))<\di(X)$. Thus $B(x,\mu r)$ is a proper subset of $X$.
\end{proof}
\begin{exam}
The Cantor space $(2^{\nn},e)$ mentioned in Example \ref{exm:sym}  is uniformly perfect 
(see e.g., \cite{DS}).
The middle-third Cantor set $(\Gamma ,d_{\Gamma})$ is also uniformly perfect.
Indeed, $(2^{\nn},e)$ and $(\Gamma,d_{\Gamma})$ are bi-Lipschitz equivalent to each other.
\end{exam}
In what follows, we will use the following observation:
\begin{lem}\label{lem:Cantor}
The middle-third Cantor set $(\Gamma,d_{\Gamma})$ is $(1/5)$-uniformly perfect.
\end{lem}
\begin{proof}
In Example \ref{exm:t}, 
we already observe that $(2^{\nn},e)$ and $(\Gamma, d_{\Gamma})$  are bi-Lipschitz equivalent 
through 
the ternary corresponding map $T:2^{\nn}\to \Gamma$. 
For all $x,y\in2^{\nn}$ we have
\begin{equation}\label{eq:lip}
\frac{3}{2}e(x,y)\le d_{\Gamma}(T(x),T(y))\le \frac{5}{2}e(x,y)
\end{equation}
(see Remark \ref{rmk:c}).
Take $a\in \Gamma $ and $r\in (0,\di(\Gamma,d_{\Gamma}))$. 
Choose $n\in \nn$ with
\[
\frac{5}{2}3^{-n}\le r<\frac{5}{2}3^{-n+1}.
\]
Since the map $T$ is homeomorphic, we can find a point $b\in \Gamma$ 
such that
\[
n=\min\{\,i\in \nn\mid (T^{-1}(a))_i\neq (T^{-1}(b))_i\,\}. 
\]
By the right hand side of (\ref{eq:lip}), we have
\[
d_{\Gamma}(a,b)\le \frac{5}{2}e(T^{-1}(a),T^{-1}(b))=\frac{5}{2}3^{-n}\le r.
\]
Hence $b\in B(a,r)$. 
By the left hand side of (\ref{eq:lip}), we have
\[
\frac{1}{5}r<\frac{3}{2}3^{-n}=\frac{3}{2}e(T^{-1}(a),T^{-1}(b))\le d_{\Gamma}(a,b).
\]
Hence $b\not\in U(a,r/5)$. 
Thus the set $B(a,r)\setminus U(a,r/5)$ is non-empty.
\end{proof}

\subsection{Product of Metric Spaces}\label{sub:prod}
For two metric spaces $(X,d_X)$ and $(Y,d_Y)$, 
we denote by $d_X\times d_Y$ the $\ell^{\mgn}$-product metric on 
$X\times Y$ defined as $d_X\times d_Y=\max\{d_X,d_Y\}$.\par
The following seems to be well-known:

\begin{lem}\label{lem:pd}
Let $(X,d_X)$ and $(Y,d_Y)$ be metric spaces. Then $(X,d_X)$ and $(Y,d_Y)$ are doubling if and only if 
$(X\times Y,d_X\times d_Y)$ is doubling.
\end{lem}\par
On the uniform disconnectedness, we have:

\begin{lem}\label{lem:pud}
Let $(X,d_X)$ and $(Y,d_Y)$ be metric spaces. 
Then $(X,d_X)$ and $(Y,d_Y)$ are uniformly disconnected 
if and only if $(X\times Y,d_X\times d_Y)$ is uniformly disconnected. 
\end{lem}
\begin{proof}
Since the uniform disconnectedness is hereditary, 
we see that if $(X\times Y,d_X\times d_Y)$ is uniformly disconnected, 
then so are $(X,d_X)$ and $(Y,d_Y)$.
Note that for any two ultrametric spaces the product is an ultrametric space.
Therefore Proposition \ref{prop:ult} leads to that if $(X,d_X)$ and $(Y,d_Y)$ are uniformly disconnected, 
then the $\ell^{\mgn}$-product metric space $(X\times Y,d_X\times d_Y)$ is uniformly disconnected.
\end{proof}

On the other hand, 
on the uniform perfectness, 
we have:
\begin{lem}\label{lem:pup}
Let $(X,d_X)$ and $(Y,d_Y)$ be bounded metric spaces. 
Assume that either $(X,d_X)$ or $(Y,d_Y)$ is uniformly perfect. 
Then $(X\times Y,d_X\times d_Y)$ is uniformly perfect.
\end{lem}
\begin{proof}
Without loss of generality, 
we may assume that $(X,d_X)$ is uniformly perfect.
By Lemma \ref{lem:ext}, 
there exists $\lambda\in (0,1)$ such that for each $x\in X$, 
and for each $r\in (0,\di(X\times Y))$, 
the subset $B(x,r)\setminus U(x,\lambda r)$ of $X$ is non-empty.
Take a point $z=(x,y)\in X\times Y$ and a number $r\in (0,\di(X\times Y))$.
Choose a point $x'\in B(x,r)\setminus U(x,\lambda r)$, and put $z'=(x',y)$. 
Then $(d_X\times d_Y)(z,z')$ is equal to $d_X(x,x')$ 
and hence it belongs to $[\lambda r,r]$.
This implies that 
the point $z'$ belongs to the subset $B(z,r)\setminus U(z,\lambda r)$ of $X\times Y$. 
\end{proof}

\begin{rmk}
In Proposition \ref{prop:nonup}, we will prove that there exist two Cantor metric spaces that are not uniformly perfect whose product metric space is uniformly perfect.
\end{rmk}

\begin{rmk}
In Lemmas \ref{lem:pd}, \ref{lem:pud} and \ref{lem:pup}, 
the $\ell^{\mgn}$-product metric $d_X\times d_Y$ can be replaced with the $\ell^p$-product metric on $X\times Y$ for any $p\in [1,\mgn)$. 
Indeed, the $\ell^{\mgn}$-product metric space $(X\times Y,d_X\times d_Y)$ is bi-Lipschitz equivalent to the $\ell^p$-product one.
\end{rmk}
\subsection{Direct Sum of Metric Spaces}
For two bounded metric spaces $(X,d_X)$ and $(Y,d_Y)$, 
we denote by $d_X\sqcup d_Y$ the metric on the disjoint union $X\sqcup Y$ defined as 
\[
(d_X\sqcup d_Y)(x,y)=
\begin{cases}
d_X(x,y) & \text{if $x,y\in X$,}\\
d_Y(x,y) & \text{if $x,y\in Y$,}\\
\max\{\di(X),\di(Y)\} & \text{otherwise.}
\end{cases}
\]
\begin{rmk}
By the Brouwer theorem \ref{Brouwer}, the direct sum of any Cantor spaces is also a Cantor space. 
\end{rmk}
From the definition of the doubling property, we have:

\begin{lem}\label{lem:sd}
Two bounded metric spaces $(X,d_X)$ and $(Y,d_Y)$ are doubling 
if and only if $(X\sqcup Y,d_X\sqcup d_Y)$ is doubling.
\end{lem}

On the uniform disconnectedness, we also have:
\begin{lem}\label{lem:sud}
Two bounded metric spaces $(X,d_X)$ and $(Y,d_Y)$ are uniformly disconnected
if and only if $(X\sqcup Y,d_X\sqcup d_Y)$ is uniformly disconnected.
\end{lem}

On the uniform perfectness, by Lemma \ref{lem:ext}, we see the following:
\begin{lem}\label{lem:sup}
Two bounded metric space $(X,d_X)$ and $(Y,d_Y)$ are uniformly perfect 
if and only if $(X\sqcup Y,d_X\sqcup d_Y)$ is uniformly perfect.
\end{lem}

\section{Telescope Spaces}\label{sec:telescope}

In this section, we introduce the notion of the telescope spaces.

\begin{df}
We say that a triple $\mathcal{B}=(B,d_B,b)$ is a \emph{telescope base} if $(B,d_B)$ is a metric space 
homeomorphic to the one-point compactification of $\nn$, and 
if $b$ is a bijective map $b:\nn\cup\{\infty\}\to B$ 
such that  $b_{\infty}$ is a unique accumulation point of $B$.
Let $\mathcal{B}=(B,d_B, b)$ be a telescope base. For $n\in \nn$ we put
\[
R_n(\mathcal{B})=\sup\{\,r\in (0,\infty)\mid U(b_n,r)=\{b_n\}\}.
\]
Note that $R_n(\mathcal{B})$ is equal to the distance in $(B,d_B)$ from $b_n$ to $B\setminus \{b_n\}$.
\end{df}
The following example of the telescope bases will be used later.   
\begin{df}\label{def:tbr}
Define a function $r:\nn\cup\{\infty\}\to \rr$ by 
$r_i=2^{-i}$, and by $r_{\infty}=0$. 
Let 
\[
R=\{\, r_i\mid i\in \nn\cup \{\infty\}\},
\]
and let $d_R$ be the metric on $R$ induced from $\rr$. 
The triple $\mathcal{R}=(R,d_R, r)$ is a telescope base. 
Note that $R_n(\mathcal{R})=2^{-n-1}$ for each $n\in \nn$. 
\end{df}
We define the telescope spaces.
\begin{df}
Let $\mathcal{X}=\{(X_i,d_i)\}_{i\in \nn}$ be a countable family of metric spaces.  
Let $\mathcal{B}=(B,d_B,b)$ be a telescope base.  
We say that $\mathcal{P}=(\mathcal{X}, \mathcal{B})$ is a \emph{compatible pair} if 
for each $n\in \nn$ we have $\di(X_n)\le R_n(\mathcal{B})$. 
Let $\mathcal{P}=(\mathcal{X}, \mathcal{B})$ be a compatible pair. 
Put 
\[
T(\mathcal{P})=\{\infty\}\sqcup \coprod_{i\in \nn}X_i, 
\]
and define a metric $d_{\mathcal{P}}$ on $T(\mathcal{P})$ by 
\[
d_{\mathcal{P}}(x,y)=
	\begin{cases}
		d_i(x,y) & \text{if $x,y\in X_i$ for some $i$,}\\
		d_B(b_i,b_j) & \text{if $x\in X_i,y\in X_j$ for some $i\neq j$, }\\
		d_B(b_{\infty},b_i) & \text{if $x=\infty, y\in X_i$ for some $i$,}\\
		d_B(b_i,b_{\infty}) & \text{if $x\in X_i, y=\infty$ for some $i$.}
	\end{cases}
\]
We call the metric space $(T(\mathcal{P}), d_{\mathcal{P}})$  the \emph{telescope space of $\mathcal{P}$}. 
\end{df}
Notice that the compatibility of $\mathcal{P}$ guarantees the triangle inequality of the metric $d_{\mathcal{P}}$ on $T(\mathcal{P})$. 
By the compatibility, we have:
\begin{lem}\label{lem:ultra}
Let $\mathcal{P}=(\mathcal{X},\mathcal{B})$ be a compatible pair. 
If $\mathcal{X}$ and $\mathcal{B}$ consist of ultrametric spaces, 
then the telescope space $(T(\mathcal{P}),d_{\mathcal{P}})$ is an ultrametric space.
\end{lem}

By the Brouwer theorem \ref{Brouwer}, we see the following:
\begin{lem}\label{lem:tlCantor}
Let $\mathcal{P}=(\mathcal{X},\mathcal{B})$ be a compatible pair. 
If the family $\mathcal{X}$ consists of Cantor spaces, 
then $(T(\mathcal{P}),d_{\mathcal{P}})$ is also a Cantor space.
\end{lem} 
From the definitions we can derive the following, which provides a method constructing a Lipschitz map between telescope spaces. 
\begin{prop}\label{prop:tllip}
Let $\mathcal{P}=(\mathcal{X}, \mathcal{B})$ and $\mathcal{Q}=(\mathcal{Y},\mathcal{C})$ be compatible pairs of 
$\mathcal{X}=\{(X_i,d_i)\}_{i\in \nn}$ and  $\mathcal{B}=(B,d_B,b)$ 
and of $\mathcal{Y}=\{(Y_i,e_i)\}_{i\in \nn}$ and $\mathcal{C}=(C,d_C,c)$, respectively.   
Let $\{f_i:X_i\to Y_i\}_{i\in \nn}$ be a family of $M$-Lipschitz maps. 
Assume that the map $\phi:B\to C$ defined by $\phi=c\circ b^{-1}$ is also $M$-Lipschitz. 
Let $F:T(\mathcal{P})\to T(\mathcal{Q})$ be a map defined by 
\[
F(x)=
	\begin{cases}
	f_i(x) & \text{if $x\in X_i$ for some $i$,}\\
	\infty& \text{if $x=\infty$.}
	\end{cases}
\]
Then $F$ is $M$-Lipschitz.
\end{prop}
Furthermore, we have: 
\begin{cor}\label{cor:lipschitz}
Under the same setting as in Proposition \ref{prop:tllip}, 
if all the maps $f_i:X_i\to Y_i$ and $\phi$ are $M$-bi-Lipschitz, 
then the map $F$ is $M$-bi-Lipschitz. 
\end{cor}
On the doubling property, we have:
\begin{prop}\label{prop:tld}
Let $\mathcal{P}=(\mathcal{X},\mathcal{B})$ be a compatible pair of 
a family $\mathcal{X}=\{(X_i,d_i)\}_{i\in \nn}$ and a telescope base $\mathcal{B}=(B,d_B,b)$ such that 
\begin{enumerate}
\item there exists $N\in \nn$ for which  each $(X_i,d_i)$ is $N$-doubling;
\item $(B,d_B)$ is doubling.
\end{enumerate}
Then the telescope space $(T(\mathcal{P}),d_{\mathcal{P}})$ is doubling. 
\end{prop}
\begin{proof}
We may assume that $(B,d_B)$ is $N$-doubling. 
We prove that $(T(\mathcal{P}),d_{\mathcal{P}})$ is $(N^2)$-doubling. 
Namely, for each  $x\in T(\mathcal{P})$ and for each $r\in(0,\infty)$, the ball $B(x,r)$ in $(T(\mathcal{P}),d_{\mathcal{P}})$ can be covered by 
at most $N^2$ closed balls with radius $r/2$. \par
Take  $n\in \nn\cup\{\infty\}$ with $x\in X_n$, where we put $X_{\infty}=\{\infty\}$. 
It suffices to consider the case where $B(x,r)$ is not contained in $X_n$.
By the definition of $d_{\mathcal{P}}$, we have 
\[
B(n,r)=\{i\in \nn\cup \{\infty\}\mid X_i\subset B(x,r)\},
\]
where $B(n,r)$ is the ball in $(B,d_B)$. 
By the definition of $d_{\mathcal{P}}$, we obtain
\begin{equation}\label{eq:aaa}
B(x,r)=\bigcup_{i\in B(n,r)}X_i. 
\end{equation}
Since $(B,d_B)$ is $N$-doubling, there exist $b_{n_1},\dots, b_{n_N}$ in $B$ with
\begin{equation}\label{eq:bbb}
B(n,r)\subset \bigcup_{i=1}^NB(b_{n_i},r/2). 
\end{equation}
For each $i\in \{1,2,\dots N\}$, take $q_i\in X_{n_i}$. 
Let 
\[
S=\{i\in\{1,\cdots, N\}\mid B(q_i,r/2)\subset X_{n_i}\}. 
\]
Notice that if $i\not\in S$, then $X_{n_i}\subset B(q_i,r/2)$. 
Hence by \eqref{eq:aaa} and \eqref{eq:bbb}, 
\begin{equation}\label{eq:ccc}
B(x,r)\setminus \left(\bigcup_{i\not\in S}B(q_i,r/2)\right)\subset \bigcup_{i\in S}X_{n_i}.
\end{equation}
For each $i\in S$, we have $\di(X_{n_i})\le r$. 
Then $X_{n_i}\subset  B(q_i,r)$.  
By the $N$-doubling property of $X_{n_i}$, we can take $q_{i1},\cdots q_{iN}$ in $X_{n_i}$ with
\[
X_{n_i} \subset \bigcup_{j=1}^NB(q_{ij}, r/2). 
\]
Hence by \eqref{eq:ccc} we obtain 
\[
B(x,r)\subset \bigcup_{i\not\in S}B(q_i,r/2)\cup \bigcup_{i\in S}\bigcup_{j=1}^NB(q_{ij}, r/2). 
\]
Therefore $(T(\mathcal{P}),d_{\mathcal{P}})$ is $(N^2)$-doubling. 
\end{proof}

On the uniform disconnectedness, we have: 
\begin{prop}\label{prop:tlud}
Let $\mathcal{P}=(\mathcal{X},\mathcal{B})$ be a compatible pair of  
a family $\mathcal{X}=\{(X_i,d_i)\}_{i\in \nn}$ and 
a telescope base $\mathcal{B}=(B,d_B,b)$ such that 
\begin{enumerate}
\item there exists $\delta\in (0,1)$ for which  each $(X_i,d_i)$ is $\delta$-uniformly disconnected;
\item $(B,d_B)$ is uniformly disconnected.
\end{enumerate}
Then the telescope space $(T(\mathcal{P}),d_{\mathcal{P}})$ is uniformly disconnected. 
\end{prop}
\begin{proof}
We may assume that $(B,d_B)$ is $\delta$-uniformly disconnected. 
By Proposition \ref{prop:ult}, 
there exists a telescope base $\mathcal{C}=(C,d_C,c)$ such that $(C,d_C)$ is an ultrametric space and
 the map $\phi=c\circ b^{-1}$ is $\delta^{-1}$-bi-Lipschitz (see Remark \ref{rmk:ud}). 
 Note that for each $i\in \nn$ we have
\[
\delta R_i(\mathcal{C})\le R_i(\mathcal{B})\le \delta^{-1} R_i(\mathcal{C}).
\]
Similary, 
there exist a family $\{(Y_i,e_i)\}_{i\in \nn}$ of ultrametric spaces 
and a family $\{f_i:X_i\to Y_i\}_{i\in \nn}$ of $\delta^{-1}$-bi-Lipschitz maps. 
Note that 
\[
\delta \di(X_i)\le \di(Y_i) \le \delta^{-1} \di(X_i).
\] 
Hence $\di(Y_i)\le \delta^{-2}R_i(\mathcal{C})$. 
Let $\mathcal{Y}=\{(Y_i,\delta^2 e_i)\}_{i\in\nn}$.  
Then $\mathcal{Y}$ and $\mathcal{C}$ are compatible. 
Since each $f_i$ is $\delta^{-3}$-bi-Lipschitz between $(X_i.d_i)$ and $(Y_i,\delta^{2}e_i)$, 
 Lemma \ref{lem:ultra} and  Corollary \ref{cor:lipschitz} complete the proof. 
\end{proof}

On the uniform perfectness, we have:
\begin{prop}\label{prop:tlup}
Assume that a countable family $\mathcal{X}=\{(X_i,d_i)\}_{i\in \nn}$ of metric spaces satisfies the following: 
\begin{enumerate}
\item the family $\mathcal{X}$ and the telescope base $\mathcal{R}=(R,d_R,r)$ defined in Definition \ref{def:tbr} are compatible; 
\item there exists $\rho\in (0,1]$ such that for each $i\in \nn$ the space $(X_i,d_i)$ is $\rho$-uniformly perfect;
\item there exists  $M\in (0,\infty)$ such that for each $i\in \nn$
\[
M\cdot 2^{-i}\le \di(X_i).
\]
\end{enumerate}
Then for the compatible pair $\mathcal{P}=(\mathcal{X},\mathcal{R})$ 
the telescope space $(T(\mathcal{P}),d_\mathcal{P})$ is uniformly perfect.
\end{prop}

\begin{proof}
By the assumption, for each $i\in \nn$, the space $X_i$ has at least two points. 
Note that $\di(T(\mathcal{P}))=2^{-1}$. 
We are going to prove that $(T(\mathcal{P}),d_\mathcal{P})$ is $\eta$-uniformly perfect, where 
\[
\eta=\min\left\{\,\frac{1}{4}, \frac{M\rho}{2}\,\right\}.
\] 
Namely, we verify that for each $x\in X$ and for each $r\in (0,2^{-1})$, 
the set $B(x,r)\setminus U(x,\eta r)$ is non-empty. \par
\begin{clm}
If $B(x,r)=T(\mathcal{P})$, then $B(x,r)\setminus U(x,\eta r)$ is non-empty.
\end{clm}

\begin{proof}
Since $\di (U(x,\eta r))<\di(T(\mathcal{P}))$, 
the set $U(x,\eta r)$ is a proper subset of $B(x,r)$.
\end{proof}

\begin{clm}
If $B(x,r)\neq T(\mathcal{P})$ and $x=\infty$,  then $B(x,r)\setminus U(x,\eta r)$ is non-empty.
\end{clm}
\begin{proof}
Take $m\in \nn$ with $r\in [2^{-m},2^{-m+1})$. 
Then $X_m\subset B(x,r)$. 
From  
$\eta r <2^{-m}$,
it follows that  $X_m\subset B(\infty,r)\setminus U(\infty,\eta r)$.
\end{proof}
\begin{clm}
If $B(x,r)\neq T(\mathcal{P})$ and $x\in X_1$, then $B(x,r)\setminus U(x,\eta r)$ is non-empty. 
\end{clm}
\begin{proof}
By the  compatibility of $\mathcal{P}$, we have $\di(X_1)\le 4^{-1}$. 
Then $r\in (0, 2\di(X_1))$. 
Since $M\le 2^{-1}$, we have $\eta \le \rho/4$. 
Thus by Lemma  \ref{lem:ext}, the set $B(x,r)\setminus U(x,\eta r)$ is non-empty. 
\end{proof}

\begin{clm}
If $B(x,r)\neq T(\mathcal{P})$ and $x\in X_n$ for some $n\ge 2$, 
then $B(x,r)\setminus U(x,\eta r)$ is non-empty.
\end{clm}
\begin{proof}
Note that $d_{\mathcal{P}}(\infty, x)=2^{-n}$ and $d_{\mathcal{P}}(x,X_1)=2^{-1}-2^{-n}$. 
Then by $B(x,r)\neq T(\mathcal{P})$, we have $2^{-n}+r<2^{-1}$. 
Hence there exists a positive integer $k\le n$ with  
$2^{-n}+r\in [2^{-k}, 2^{-k+1})$.
We divide the present situation  into the following two cases.\par
First assume $k\le n-1$. 
Take $y\in X_k$. Then we have 
\[
d_{\mathcal{P}}(x,y)=2^{-k}-2^{-n}\ge 2^{-k}-2^{-k-1} =2^{-k-1}, 
\]
and 
\[
r<2^{-k+1}-2^{-n}<2^{-k+1}.
\]
Hence 
$\eta r\le r/4<d_{\mathcal{P}}(x,y)$. 
Therefore $y\in B(x,r)\setminus U(x,\eta r)$. \par
Second assume $k=n$.
In this case, we have $r< 2^{-n}$, 
and hence 
\[
r\le \frac{1}{M}\di(X_n).
\] 
By Lemma \ref{lem:ext} and $\eta\le (M\rho)/2$, the set $B(x,r)\setminus U(x,\eta r)$ is non-empty. 
\end{proof}
This finishes the proof of Proposition \ref{prop:tlup}.
\end{proof}

\section{Spike Spaces}\label{sec:spike}
In this section, we study the existence of the spike spaces defined in Definition \ref{def:spike} for the quasi-symmetric invariant properties, $D$, $UD$ and  $UP$.
\par
First we study the existence of a $D$-spike Cantor metric space. 
Before doing that, we give a criterion of the doubling property. 
\begin{df}\label{def:discrete}
For $n\in \nn$ and for $l\in (0,\infty)$, we say that a metric space $(X,d)$ is \emph{$(n,l)$-discrete}  
if $\card(X)=n$ and the metric $d$ satisfies 
\[
d(x,y)=
	\begin{cases}
	0 & \text{if $x=y$,}\\
	l & \text{if $x\neq y$}
	\end{cases}
\] for all $x,y\in X$.
A metric space $(X,d)$ is \emph{$n$-discrete} if it is $(n,l)$-discrete for some $l$.
\end{df}

\begin{lem}\label{lem:dsc}
If for each $n\in \nn$ a metric space $(X,d)$ has an $n$-discrete subspace, then $(X,d)$ is not doubling.
\end{lem}
\begin{proof}
Let $(D_n,e_n)$ be an $n$-discrete subspace of $(X,d)$. 
Choose $l_n\in (0,\infty)$ such that $(D_n,e_n)$ is $(n,l_n)$-discrete. 
For every $p\in D_n$, the subspace  $D_n$ is contained in $ B(p,l_n)$, and $D_n$ is $(l_n/2)$-separated in $B(p,l_n)$. 
Since $\card(D_n)=n$, by Lemma \ref{lem:separated}, the space $(X,d)$ is not doubling.
\end{proof}

We construct a $D$-spike Cantor metric space. 
\begin{prop}\label{prop:dspike}
There exists a $D$-spike Cantor metric space  of type $(0,1,1)$.
\end{prop}
\begin{proof}
For each $n\in \nn$, take disjoint $n$ copies $\Gamma_1, \dots \Gamma_n$ of 
the middle-third Cantor set $\Gamma$, 
and define a set $Z_n$ by 
\[
Z_n=\coprod_{i=1}^n\Gamma_i,
\]
and define a metric $e_n$ on $Z_n$ by
\[
e_n(x,y)=
	\begin{cases}
	d_{\Gamma_i}(x,y) & \text{if $x,y\in \Gamma_i$ for some $i$,}\\
	1 & \text{otherwise.}
	\end{cases}
\]
Note that for each $n\in \nn$, the space $(Z_n,e_n)$ is a Cantor space.
The family $\mathcal{Z}=\{(Z_i,\,2^{-i-1}\cdot e_i)\}_{i\in \nn}$ and the telescope base $\mathcal{R}$ defined in Definition \ref{def:tbr} form a compatible pair. \par
Let $\mathcal{P}=(\mathcal{Z},\mathcal{R})$. 
By Lemma \ref{lem:tlCantor}, the telescope space $(T(\mathcal{P}),d_{\mathcal{P}})$ is a Cantor space. 
We first show that $(T(\mathcal{P}),d_{\mathcal{P}})$ is a $D$-spike space. 
For each neighborhood $N$ of $\infty$, there exists $k\in \nn$ such that $n>k$ implies $Z_n\subset N$. 
Then $N$ has an $n$-discrete subspace for all sufficiently large $n$. 
By Lemma \ref{lem:dsc}, the subspace $N$ is not doubling. 
Since for each $i\in \nn$ the space $(Z_i,e_i)$ is doubling,  $S_D((T(\mathcal{P}),d_{\mathcal{P}}))=\{\infty\}$. 
Hence $(T(\mathcal{\mathcal{P}}),d_{\mathcal{P}})$ is a $D$-spike space. 
\par
We next show that $(T(\mathcal{\mathcal{P}}),d_{\mathcal{P}})$ has type $(0,1,1)$.
Since $(\Gamma,d_{\Gamma})$ is $\delta$-uniformly disconnected for  $\delta\in (0,3/5)$ (see Remark \ref{rmk:c}),  
each $(Z_i,e_i)$ is $\delta$-uniformly disconnected. 
The space $(R,d_R)$ is uniformly disconnected. 
Then Proposition  \ref{prop:tlud} implies that 
$(T(\mathcal{P}), d_{\mathcal{P}})$ is an ultrametric space. 
By Proposition \ref{prop:ult}, the space $(T(\mathcal{\mathcal{P}}),d_{\mathcal{P}})$ is uniformly disconnected. 
By Lemmas \ref{lem:ext} and \ref{lem:Cantor}, for each $i\in \nn$, 
the space $(Z_i,d_i)$ is $(1/20)$-uniformly perfect.  
Therefore by Lemma \ref{prop:tlup}, the  space $(T(\mathcal{P}),d_{\mathcal{P}})$ is  uniformly perfect.
\end{proof}

Second we study the existence of a $UD$-spike Cantor metric space.
To do this, we need the following:
\begin{lem}\label{lem:ssum}
For $\rho\in (0,\infty)$,  
let $\{(X_i,d_i)\}_{i=0}^n $ be a finite family of compact subspaces of $(\rr,d_{\rr})$ satisfying the following: 
\begin{enumerate}
\item each $(X_i,d_i)$ is $\rho$-uniformly perfect;
\item $\di(X_i)=1$ for all $i$;
\item $d_{\rr}(X_i,X_j)=1$ for all disjoint $i,j$. 
\end{enumerate}
Then the subspace $\bigcup_{i=0}^nX_i$ of $\rr$
is $\min\{1/3,\rho/4\}$-uniformly perfect.
\end{lem}
\begin{proof}
We may assume that 
\[
\{2i,2i+1\}\subset X_i\subset [2i,2i+1]
\]
 for each $i$.
Set $\nu=\min\{1/3,\rho/4\}$. 
Take $x\in X$ and $r\in (0,\di(X))$.
We show that $B(x,r)\setminus U(x,\nu r)$ is non-empty. 
If $B(x,r)=X$, then $\di(U(x,\nu r))$ is smaller than $(2/3)\di(X)$; 
in particular, the set $B(x,r)\setminus U(x,\nu r)$ is non-empty. 
Suppose $B(x,r)\neq X$. 
Then either $0<x-r$ or $x+r<2n+1$ holds.
The case of $0<x-r$ can be reduced to the case of $x+r<2n+1$ 
through the map defined by $t\mapsto -t+2n+1$. 
Hence it is enough to consider the case of $x+r<2n+1$. 
Let $m$ be an integer with $x\in X_m$.
Take a positive integer $k$ with $x+r\in [k-1,k)$ so that $k\ge 2m+1$. 
If $k-(2m+1)\ge 2$, then
\[
x+\nu r\le x+\frac{1}{3}r< \frac{2}{3}(2m+1)+\frac{1}{3}k<k-1,
\]
and hence $k-1\in B(x,r)\setminus U(x,\nu r)$. 
If $k-(2m+1)=1$, then 
$r<2\di(X_m)$, and hence Lemma \ref{lem:ext} implies that  
$B(x,r)\setminus U(x,\nu r)$ is non-empty. 
If $k=2m+1$, then $r<\di(X_m)$, and hence 
the $\rho$-uniformly perfectness of $X_m$ and $\nu\le \rho$ lead to the desired conclusion.
This finishes the proof. 
\end{proof}
For a subset $S$ of $\rr$, and for real numbers $a,b$,  we denote by $aS+b$ the set $\{\,ax+b\mid x\in S\,\}$.
\par
We construct a $UD$-spike Cantor metric space. 
\begin{prop}\label{prop:udspike}
There exists a $UD$-spike Cantor metric space of type $(1,0,1)$.
\end{prop}
\begin{proof}
For each $n\in \nn$, we define a subset $F_n$ of $\rr$ by 
\[
F_n=\frac{2^{-n-1}}{(2n-1)}\left(\bigcup_{i=0}^{n-1}(2i+\Gamma)\right),
\] 
and we denote by $e_n$ the metric on $F_n$ induced from $d_{\rr}$, 
where $\Gamma$ is the middle-third Cantor set.
Note that  $F_n$ has a non-trivial $(1/(2n-1))$-chain.
By Lemma \ref{lem:Cantor}, the middle-third Cantor set $(\Gamma, d_{\Gamma})$ is $(1/5)$-uniformly perfect.
Using Lemma \ref{lem:ssum}, 
we see that the space $(F_n,e_n)$ is $(1/20)$-uniformly perfect. 
Note that $\di(F_n)=2^{-n-1}$. 
Then the family $\mathcal{F}=\{(F_i,e_i)\}_{i\in \nn}$ and the telescope base $\mathcal{R}$ defined in Definition \ref{def:tbr} is compatible.\par 
Let $\mathcal{P}=(\mathcal{F},\mathcal{R})$. 
By Lemma \ref{lem:tlCantor}, the telescope space $(T(\mathcal{P}),d_{\mathcal{P}})$ is a Cantor space. 
We  prove that $(T(\mathcal{P}),d_{\mathcal{P}})$ is a desired space. 
From the construction of $(T(\mathcal{P}),d_{\mathcal{P}})$, 
it follows that each neighborhood of $\infty$ has a non-trivial $(1/(2n-1))$-chain for every sufficiently large $n$. 
Hence $(T(\mathcal{P}),d_{\mathcal{P}})$ is not uniformly disconnected.  
A small enough neighborhood of an arbitrary point except $\infty$ is 
bi-Lipschitz equivalent to some open set of $(\Gamma,d_{\Gamma})$. 
Hence $S_{UD}(T(\mathcal{P}),d_{\mathcal{P}})=\{\infty\}$. 
This implies that $(T(P),d_{\mathcal{P}})$ is a $UP$-spike space. 
By Propositions \ref{prop:tld} and \ref{prop:tlup},
 the space $(T(\mathcal{P}),d_{\mathcal{P}})$ is doubling and uniformly perfect.
Therefore $(T(\mathcal{P}),d_{\mathcal{P}})$ has type $(1,0,1)$. 
\end{proof}
Third we study the existence of a $UP$-spike Cantor metric space. 
\begin{prop}\label{prop:upspike1}
There exists a compatible pair $\mathcal{P}$ such that 
the space $(T(\mathcal{P}),d_{\mathcal{P}})$ is  a $UP$-spike Cantor metric space of type $(1,1,0)$ 
satisfying the following:  
For each $\rho \in (0,\infty)$ there exists $r\in (0,\di (T(\mathcal{P})))$ with  
\[
B(\infty, r)\setminus U(\infty,\rho r)=\emptyset;
\]
in particular, $S_{UP}(T(\mathcal{P}),d_{\mathcal{P}})=\{\infty\}$. 
\end{prop}
\begin{proof}
Define a function $v:\nn\cup\{\infty\}\to \rr$ by $v_n=(n!)^{-1}$ if $n\in \nn$,
 and $v_{\infty}=0$. Put $V=\{0\}\cup \{\,v_n\mid n\in \nn\,\}$. 
Let $d_V$ be the metric on $V$ induced from $d_{\rr}$. 
Then $\mathcal{V}=(V,d_V,v)$ is a telescope base. 
For each $i\in \nn$, let
\[
G_i=\frac{1}{(i+1)!} \Gamma, 
\] 
and let $d_i$ be the metric on $G_i$ induced from $d_{\rr}$. 
Since $R_i(\mathcal{V})\ge 1/(i+1)!$, the pair of $\mathcal{G}=\{(G_i,d_i)\}_{i\in \nn}$ and $\mathcal{V}$ is compatible. 
\par
Let $\mathcal{P}=(\mathcal{G},\mathcal{V})$. 
By Lemma \ref{lem:tlCantor}, the telescope space $(T(\mathcal{P}),d_{\mathcal{P}})$ is a Cantor space. 
We prove that $(T(\mathcal{P}),d_{\mathcal{P}})$ is a desired space. 
For each $\rho\in (0,1)$, take $n\in \nn$ with $\rho >2/(n+1)$. 
Then by the definition of $d_{\mathcal{P}}$ we have 
\[
B\left(\infty,\frac{1}{2n!}\right)=B\left(\infty,\frac{\rho}{2n!}\right).
\]
Hence $\infty\in S_{UD}(T(\mathcal{P}),d_{\mathcal{P}})$. 
If $x\neq \infty$, then $x$ has a neighborhood which is Lipschitz equivalent to the middle-third Cantor set $(\Gamma,d_{\Gamma})$. 
This implies that $S_{UD}(T(\mathcal{P}),d_{\mathcal{P}})=\{\infty\}$. 
Therefore 
$(T(\mathcal{P}),d_{\mathcal{P}})$ is a $UD$-spike space. 
Since $(V,d_V)$ is doubling and uniformly disconnected, 
by Propositions \ref{prop:tld} and \ref{prop:tlud}, 
the space $(T(\mathcal{P}),d_{\mathcal{P}})$ is doubling and uniformly disconnected. 
Therefore $(T(\mathcal{P}),d_{\mathcal{P}})$ has type $(1,1,0)$. 
\end{proof}

\begin{rmk}\label{rmk:upspike2}
There exists a compatible pair $\mathcal{P}$ such that 
the space $(T(\mathcal{P}),d_{\mathcal{P}})$ is  a $UP$-spike Cantor metric space of type $(1,1,0)$ 
satisfying: 
\begin{enumerate}
\item $S_{UP}(X,d)=\{\infty\}$;
\item there exists $\rho\in (0,1]$ such that for each  $r\in (0,\di (T(\mathcal{P})))$
\[
B(\infty, r)\setminus U(\infty,\rho r)\neq\emptyset. 
\]
\end{enumerate}
\end{rmk}

\begin{rmk}\label{rmk:multispike}
Using the constructions of $D$, $UD$ or $UP$-spike Cantor metric spaces discussed above, 
for each exotic type $(u,v,w)$,  
we can obtain the Cantor metric space $(T(\mathcal{P}),d_{\mathcal{P}})$ of type $(u,v,w)$ such that 
$S_P(T(\mathcal{P}),d_{\mathcal{P}})=\{\infty\}$ 
for all $P\in \{D,UD,UP\}$ with $T_P(T(\mathcal{P}),d_{\mathcal{P}})=0$. 
\end{rmk}

\section{Abundance of Exotic Cantor Metric Spaces }\label{sec:prf}
In this section, we prove Theorem \ref{thm:many}.
\subsection{Leafy Cantor Spaces}
Let $X$ be a topological space, 
and let $A$ be a subset of $X$. 
We denote by $D(A)$ the derived set of $A$ consisting of all accumulation points of $A$ in $X$. 
For $k\in \nn\cup \{0\}$, we denote by $D^k(A)$ the $k$-th derived set of $A$ inductively defined as $D^{k}(A)=D(D^{k-1}(A))$, where $D^0(A)=A$. 
Recall that $A$ is perfect in $X$ if and only if  $D(A)=A$. 
\begin{df}
We say that  $x\in X$ is a \emph{perfect point of $X$} if there exists a perfect neighborhood in $X$ of $x$. 
We denote by $P(X)$ the set of all perfect points of $X$,
and call $P(X)$ the \emph{perfect part of $X$}.
\end{df}

Notice that $X$ is perfect in $X$ if and only if $P(X)=X$.

\begin{df}
We say that $X$ is \emph{anti-perfect} if $P(X)$ is empty; in other words, each open set of $X$ has an isolated point.
\end{df}
We introduce the following:
\begin{df}
For  $n\in \nn$, 
we say that a topological space $X$ is \emph{an $n$-leafy Cantor space} if  $X$ satisfies the following: 
\begin{enumerate}
\item $X$ is a $0$-dimensional compact metrizable space;
\item $D^k(X)$ is anti-perfect for all $k<n$;
\item $D^n(X)$ is a Cantor space.
\end{enumerate}
\end{df}

In order to prove the existence of leafy Cantor spaces,  
we refer to a construction of the middle-third Cantor set by using the iterating function system. 
\begin{df}\label{def:lambda}
Let $S$ be a compact subset of $\rr$ with $(1/3)S\subset S$ and $\di(S)\le 2^{-1}$. 
Let $f_{0}(x)=(1/3)x$ and  $f_{1}(x)=(1/3)x+(2/3)$. 
We inductively define a family $\{V_i(S)\}_{i\in \nn\cup\{0\}}$ of subsets of $\rr$ by
\begin{align*}
V_0(S)=(-S)\cup (1+S),\quad V_{i+1}(S)=f_{0}(V_i(S))\cup f_1(V_i(S)).
\end{align*}
Put 
\[
L(S)=\bigcup_{i\in \nn}V_i(S),
\]
and $\Lambda(S)=\cl_{\rr}(L(S))$, 
where $\cl_{\rr}$ is the closure operator in $\rr$.
\end{df}
\begin{rmk}
The construction in Definition \ref{def:lambda} contains the middle-third Cantor set. 
Namely, we have  $\Lambda(\{0\})=\Gamma$.
\end{rmk}

By definition, we have the following:
\begin{lem}\label{lem:dls}
Let $S$ be a compact subset of $\rr$ with $(1/3)S\subset S$ and $\di(S)\le 2^{-1}$. 
Then for each $n\in \nn\cup\{0\}$ we have
\[
D^n(\Lambda(S))=\Gamma\cup L(D^n(S)).
\]
\end{lem}
We verify the existence of leafy Cantor spaces.
\begin{prop}\label{prop:leaf}
For every $n\in \nn$, there exists an $n$-leafy Cantor space.
\end{prop}
\begin{proof}
Put $S=\{0\}\cup \{3^{-i}\mid i\in \nn\cup \{0\}\}$. 
We inductively define a family $\{S_i\}_{i\in \nn}$ of subsets of $\rr$ by 
\begin{align*}
S_1=S, \quad S_{i+1}=S_i+S.
\end{align*}
Then  $D(S_1)=\{0\}$. 
For each $n\in \nn$, we have  $D(S_n)=S_{n-1}$, and hence $D^n(S_n)=\{0\}$.
Let $T_n=(1/2n)\cdot S_n$. Note that $T_n$ is a compact subset of $\rr$ and satisfies
$(3^{-1})\cdot T_n\subset T_n$ and $\di(T_n)=2^{-1}$. 
Then we can define the space  $\Lambda(T_n)$ for $T_n$ (see Definition \ref{def:lambda}).
By Lemma \ref{lem:dls}, we conclude that $\Lambda(T_n)$ is surely an $n$-leafy Cantor space.
\end{proof}

\subsection{Topological Observation}
For a topological space $X$, 
let $\mathcal{C}(X)$ be the set of all closed sets in $X$, 
and let $\mathcal{H}(X)$ be the quotient set $\mathcal{C}(X)/\!\approx$ of $\mathcal{C}(X)$ 
divided by $\approx$, where 
the symbol $\approx$ denotes the homeomorphic relation on $\mathcal{C}(X)$. \par 
\begin{df}\label{def:xi}
For each $n\in \nn$, by Proposition \ref{prop:leaf}, 
there exists an $n$-leafy Cantor space $\Lambda_n$. 
We may assume $ \Lambda_n \subset [2^{-2n}, 2^{-2n+1}]$. 
Note that if $n\neq m$, then $\Lambda_n\cap \Lambda_m$ is empty. 
Let $I$ be the set of all points $x\in 2^{\nn}$ such that $\card(\{\,i\in \nn\mid x_i=1 \})$ is infinite.  
Note that $\card(I)=2^{\aleph_0}$. 
For each $x\in I$, we define 
\[
\Xi(x)=\{0\}\cup \bigcup_{x_i=1}\Lambda_i.
\]
Then $\Xi(x)$ is a $0$-dimensional compact metrizable space. 
Since a $0$-dimensional compact metrizable space can be topologically embedded into the middle-third Cantor set $\Gamma$ (\cite{Ury}, see e.g., \cite[Theorem 2 in \S 26. IV]{Ku}), 
the space $\Xi(x)$ can be considered as the closed subspace of $\Gamma$. 
Thus  we obtain a map $\Xi:I\to \mathcal{C}(\Gamma)$ by assigning each point $x\in I$ to the space  $\Xi(x)$. 
\end{df}
\begin{rmk}\label{rmk:anti}
Since each $\Lambda_i$ is anti-perfect,  so is $\Xi(x)$ for each $x\in I$. 
\end{rmk}
The following proposition is a key to prove Theorem \ref{thm:many}. 
\begin{prop}\label{prop:xiinj}
The map $[\Xi]:I\to \mathcal{H}(\Gamma)$ defined by
 $[\Xi](x)=[\Xi(x)]$ is injective, where $[\Xi(x)]$ stands for the equivalent class of $\Xi(x)$. 
\end{prop}
\begin{proof}
We inductively define a family $\{A_i\}_{i\in \nn}$ of topological operations by
\begin{align*}
A_1(X)=P(D(X)),\quad A_{i}(X)=P(D(D^{i-1}(X)\setminus P(D^{i-1}(X))))
\end{align*}
if $i\ge 2$. 
By definition, if $X$ and $Y$ are homeomorphic, then so are $A_i(X)$ and $A_i(Y)$ for each $i\in \nn$. 
\par
If $i\in \nn$ satisfies $x_i=1$, then the space $\Lambda_i$ is an open set in $\Xi(x)$. 
Note that  for each $k\in \nn$, we have 
\[
D^k(\Xi(x))=\{0\}\cup \bigcup_{x_i=1}D^k(\Lambda_i). 
\]
Since each $\Lambda_i$ is an $i$-leafy Cantor space, any neighborhood of  $0$ in $D^k(\Xi(x))$ has an isolated point, and hence 
\[
P(D^k(\Xi(x)))=\bigcup_{x_i=1, i\le k}D^i(\Lambda_i).
\]
This implies that  if $k\ge 2$, then
\[
D(D^{k-1}(X)\setminus P(D^{k-1}(X)))=\{0\}\cup \bigcup_{x_i=1, i\ge k}D^k(\Lambda_i).
\]
From the argument discussed above, it follows that 
if $n\in \nn$ satisfies $x_n=1$, then $A_n(\Xi(x))=D^n(\Lambda_n)$, and hence $A_n(\Xi(x))\approx \Gamma$; 
if $n\in\nn$ satisfies $x_n=0$, then $A_n(\Xi(x))=\emptyset$.
Therefore, if $x,y\in I$ satisfy $x\neq y$, then $\Xi(x)\not\approx \Xi(y)$. 
Namely, the map $[\Xi]:I\to \mathcal{H}(\Gamma)$ is injective. 
\end{proof}
As an application of Proposition \ref{prop:xiinj}, we have:
\begin{cor}\label{cor:con}
For the middle-third Cantor set $\Gamma$, we have
\[
\card(\mathcal{H}(\Gamma))=2^{\aleph_0}.
\]
\end{cor}
\begin{proof}
From the second countability of $\Gamma$, we have $\card(\mathcal{H}(\Gamma))\le 2^{\aleph_0}$. 
By Proposition \ref{prop:xiinj}, 
we conclude $\card(\mathcal{H})\ge 2^{\aleph_0}$. 
\end{proof}

Since  an uncountable polish space contains a Cantor space as a subspace 
(see e.g., \cite[Corollary 6.5]{Ke}), 
we obtain:
\begin{cor}
Let $X$ be an uncountable polish space. 
Then we have $\card(\mathcal{H}(X))=2^{\aleph_0}$. 
\end{cor}

\subsection{Proof of Theorem \ref{thm:many}}
Let $P$ be a property of metric spaces, and let $(X,d)$ be a metric space. 
Recall that $S_P(X,d)$ is the set of all points in $X$ of which no neighborhoods satisfy $P$ (see Definition \ref{def:sp}). 
\begin{lem}\label{lem:sqcup}
For a property $P$ of metric spaces, and for metric spaces $(X, d)$ and $(Y,e)$, we have 
\[
S_P(X\sqcup Y, d\sqcup e)=S_P(X,d)\sqcup S_P(Y,e). 
\]
\end{lem}

Lemmas \ref{lem:sd}, \ref{lem:sud} and \ref{lem:sup} imply:
\begin{lem}\label{lem:land}
Let $(X,d_X)$ and $(Y,d_Y)$ be metric spaces of type $(v_1,v_2,v_3)$ and of type $(w_1,w_2,w_3)$, respectively. 
Then $(X\sqcup Y,d_X\sqcup d_Y)$ has type $(v_1\land w_1, v_2\land w_2, v_3\land w_3)$. 
\end{lem}
This leads to the following:
\begin{lem}\label{lem:type}
For each $(v_1,v_2,v_3)\in \{0,1\}^3$, 
there exists a Cantor metric space of type $(v_1,v_2,v_3)$. 
\end{lem}
\begin{proof}
Notice that 
the set  $\{(1,1,1), (1,0,1), (0,1,1), (1,1,0)\}$ generates $\{0,1\}^3$ by the minimum operation $\land$. 
By Propositions \ref{prop:dspike}, \ref{prop:udspike}, and \ref{prop:upspike1}, 
we already obtain Cantor metric spaces whose types are $(1,1,1)$, $(0,1,1)$, $(1,0,1)$ or $(1,1,0)$. 
Therefore Lemma \ref{lem:land} completes the proof.
\end{proof}

By Lemmas \ref{lem:pd} and \ref{lem:pud}, we see the following:
\begin{lem}\label{lem:time}
Let $(A,d_A)$ be a closed metric subspace of $(\Gamma,d_{\Gamma})$. 
Let $P$ stand for either  D or UD.  Let $(X,d_X)$ be  a $P$-spike Cantor metric space
with $S_P(X,d_X)=\{x\}$. 
Then 
$(X\times A,d_X\times d_A)$ is a Cantor space such that  
$S_{P}(X\times A,d_X\times d_A)=\{x\}\times A$. In particular, $S_P(X\times A,d_X\times d_A)\approx A$.
\end{lem}
Let $\mathscr{H}=\{\,\Xi(x)\mid x\in I\,\}$, where 
$\Xi: I\to \mathcal{C}(\Gamma)$ is the map defined in Definition \ref{def:xi}. 
 Then $\mathscr{H}$ satisfies the following:
\begin{enumerate}
\item every $A\in \mathscr{H}$ is anti-perfect (see Remark \ref{rmk:anti}); in other words, the set of all isolated points of $A$ is dense in $A$; 
\item if $A,B\in \mathscr{H}$ satisfy $A\neq B$, then $A\not\approx B$ (see Proposition \ref{prop:xiinj}).  
 \end{enumerate}
Since $\card(I)=2^{\aleph_0}$, we have $\card(\mathscr{H})=2^{\aleph_0}$. 

\begin{lem}\label{lem:timeup}
Let $(X,d_X)$ be a $UP$-spike Cantor metric space mentioned in Proposition \ref{prop:upspike1}. 
Then for every $A\in \mathscr{H}$, the space $(X\times A,d_X\times d_A)$ is a Cantor space such that  
$S_{UP}(X\times A,d_X\times d_A)=\{\infty\}\times A$. 
In particular, $S_{UP}(X\times A,d_X\times d_A)\approx A$. 
\end{lem}
\begin{proof}
Each point in $X$ except $\infty$ has a uniformly perfect neighborhood. 
By Lemma  \ref{lem:pup}, each point in $(X\setminus \{\infty\})\times  A$ has a uniformly perfect neighborhood. 
If $y\in A$ is an isolated point of $A$, 
then for sufficiently small $r\in (0,\infty)$ 
the closed ball $B((\infty,y),r)$ in $X\times A$ is isometric to $B(\infty,r)$ in $X$. 
In this case, each neighborhood of $(\infty,y)$ is not uniformly perfect,  
and hence $(\infty,y)\in S_{UP}(X\times A,d_X\times d_A)$. 
If $y$ is an accumulation point of $A$, 
then a neighborhood $U$ of $(\infty,y)$ contains a point $(\infty,z)$ for some isolated point $z$ in $A$. 
Thus $U$ is not uniformly perfect, and hence $(\infty,y)\in S_{UP}(X\times A,d_X\times d_A)$. 
Therefore $S_{UP}(X\times A,d_X\times d_A)=\{\infty\}\times A$. 
\end{proof}

\begin{proof}[Proof of Theorem \ref{thm:many}]
By Propositions \ref{prop:dspike} and \ref{prop:udspike}, 
we can take a $D$-spike Cantor metric space $(F,d_F)$ of type $(0,1,1)$, 
and 
 a $UD$-spike Cantor metric space $(G,d_G)$ of type $(1,0,1)$. 
Let $(H,d_H)$ be a $UP$-spike Cantor metric space of type $(1,1,0)$ stated in Proposition \ref{prop:upspike1}. 
For each $(u,v,w)\in \{0,1\}^3$, 
we can certainly take a Cantor metric space $(L_{uvw},d_{uvw})$ of type $(u,v,w)$ as seen in Lemma \ref{lem:type}.  
We define three maps 
$f_{0vw}:\mathscr{H} \to \mathscr{M}(0,v,w)$, 
$g_{u0w}:\mathscr{H}\to \mathscr{M}(u,0,w)$ and 
$h_{uv0}:\mathscr{H}\to \mathscr{M}(u,v,0)$ as follows:
\begin{align*}
&f_{0vw}(A)=\mathcal{G}((F\times A)\sqcup L_{1vw} ,(d_F\times d_A)\sqcup d_{1vw}),\\
&g_{u0w}(A)=\mathcal{G}((G\times A)\sqcup L_{u1w} ,(d_G\times d_A)\sqcup d_{u1w}),\\
&h_{uv0}(A)=\mathcal{G}((H\times A)\sqcup L_{uv1} ,(d_H\times d_A)\sqcup d_{uv1}). 
\end{align*}
By Lemmas \ref{lem:sqcup}, \ref{lem:time} and \ref{lem:timeup}, we have 
\begin{align*}
S_D(f_{0vw}(A))\approx A,\quad S_{UD}(g_{u0w}(A))\approx A,\quad S_{UP}(h_{uv0}(A))\approx A. 
\end{align*}
Since the operators $S_D$, $S_{UD}$ and $S_{UP}$ 
are quasi-symmetric invariants (see Remark \ref{rmk:qsinv}), 
the maps $f_{0vw}$, $g_{u0v}$ and $h_{uv0}$ are injective.  
Therefore for each exotic type $(u,v,w)\in \{0,1\}^3$ we have 
\[
\card(\mathscr{M}(u,v,w))\ge 2^{\aleph_0}.
\]
\par
In general, for a separable space $X$, 
the cardinality of the set of all continuous real-valued functions on $X$ is at most $2^{\aleph_0}$.
Hence the set of all metrics on the middle-third Cantor set compatible with the  Cantor space topology has cardinality at most $2^{\aleph_0}$. 
Therefore we have
\[
\card(\mathscr{M}(u,v,w))\le 2^{\aleph_0}.
\]
This completes the proof of Theorem \ref{thm:many}. 
\end{proof}
\begin{rmk}
For each $(u,v,w)\in \{0,1\}^3$ and for each $A\in \mathscr{H}$, by taking a direct sum of spaces in $f_{011}(A)$, $g_{101}(A)$ or $h_{110}(A)$, we can obtain a Cantor metric space $(X,d)$ with 
\[
S_P(X,d)=A
\]
for all failing property $P\in\{D, UD, UP\}$ of $(u,v,w)$, where $f_{011}$, $g_{101}$, $h_{110}$ are the maps appeared in the proof of Theorem \ref{thm:many}. 
\end{rmk}

\section{Sequentially Metrized Cantor Spaces}\label{sec:smcs}
In this section, 
we generalize the construction of the symbolic Cantor sets studied by David and Semmes \cite{DS}. 
The same generalized construction is discussed by Semmes in \cite{S1, S2} in other contexts. 
\par
\subsection{Generalities}
We take a valuation map
 $v:2^{\nn}\times 2^{\nn}\to \nn\cup \{\infty\}$ defined as 
 \[
 v(x,y)=
    \begin{cases}
     \min\{\, n\in \nn\mid x_n\neq y_n\} & \text{if $x\neq y,$}\\
     \mgn		& \text{if $x=y$.}
    \end{cases}
 \]
\begin{df}
We say that a positive sequence $\alpha:\nn\to (0,\mgn)$ is \emph{shrinking}
 if  $\alpha$ is monotone non-increasing
 and if $\alpha$ converges to $0$.
For a shrinking sequence $\alpha$, 
we define a metric 
$d_{\alpha}$ on 
$2^{\nn}$ by
\[
d_{\alpha}(x,y)=
  \begin{cases}
  \alpha(v(x,y)) & \text{if $x\neq y,$}\\
  0		& \text{if $x=y$.}
  \end{cases}
\]
We call $(2^{\nn},d_{\alpha})$ the \emph{sequentially metrized Cantor space metrized by $\alpha$.} 
\end{df}

\begin{lem}
Let $\alpha$ be a shrinking sequence.
Then $(2^{\nn},d_{\alpha})$ is an ultrametric space. In particular, 
$(2^{\nn},d_{\alpha})$ is uniformly disconnected.
\end{lem}
\begin{proof}
To prove the first half, it is enough to show that $d_{\alpha}$ satisfies the ultrametric triangle inequality.
For all $x,y,z\in 2^{\nn}$, we have 
$\min\{v(x,z),v(z,y)\}\le v(x,y)$;  
in particular, 
\[
d_{\alpha}(x,y)\le \max\{d_{\alpha}(x,z), d_{\alpha}(z,y)\}.
\]
Hence $(2^{\nn},d_{\alpha})$ is an ultrametric space.
The second half follows from Proposition \ref{prop:ultraud}.
\end{proof}

The Brouwer theorem \ref{Brouwer} tells us that 
the space $(2^{\nn},d_{\alpha})$ is a Cantor space for any shrinking sequence $\alpha$.

The doubling property of $(2^{\nn},d_{\alpha})$ depends on how the shrinking sequence $\alpha$ decreases. 


\begin{lem}\label{lem:db}
Let $\alpha$ be a shrinking sequence.
Then $(2^{\nn},d_{\alpha})$ is doubling if and only if there exists $N\in \nn$ such that for all $k\in \nn$ we have
\begin{equation}\label{eq:dbp}
\card(\{\, n\in \nn\mid \alpha(k)/2\le \alpha(n) \le \alpha(k)\,\})\le N.
\end{equation}
\end{lem}
\begin{proof}
For $i\in \nn$, we put $J_{\alpha}(i)=\{\, n\in \nn\mid \alpha(i)/2\le \alpha(n) \le \alpha(i)\,\}$.\par
First we show that the condition \eqref{eq:dbp} for some $N$ implies the doubling property. 
Take $x\in 2^{\nn}$ and $r\in (0,\mgn)$.
Choose $k\in \nn$ with $r\in [\alpha(k),\alpha(k-1))$.
Note that $B(x,r)=B(x,\alpha(k))$, and 
\[
B(x,\alpha(k))=\{\,y\in 2^{\nn}\mid k\le v(x,y)\,\}.
\]
Let $S_{k+N}$ be the set of all points $z\in 2^{\nn}$ such that $z_i=0$ for all $i>k+N$.
Then $B(x,r)\cap S_{k+N}$ consists of $2^{N+1}$ elements.
For every $y\in B(x,r)$, there exists $z\in B(x,r)\cap S_{k+N}$ with $k+N\le v(y,z)$.
Since $k+N\not\in J_{\alpha}(k)$,
we have 
\[
d_{\alpha}(y,z)\le \alpha(k+N)<\frac{\alpha(k)}{2}\le \frac{r}{2}.
\]
This implies that $B(x,r)$ can be covered by at most $2^{N+1}$ balls with radius $r/2$.
Hence $(2^{\nn},d_{\alpha})$ is doubling.\par
Next, to show the contrary, we assume that for each $N\in \nn$ there exists $k\in \nn$ such that
$\card (J_{\alpha}(k))> N.$
Note that $k+1,\dots,k+N$ are contained in $J_{\alpha}(k)$ 
since so is $k$.
For each $i\in\{1,\dots, N\},$ we define a point $x^{(i)}=\{x_{i,n}\}_{n\in \nn}$ in $B(0,\alpha(k))$ by 
\[
x_{i,n}=
  \begin{cases}
  0 & \text{if $n\neq k+i$,}\\
  1 & \text{if $n=k+i$.}
  \end{cases}
\]
For all distinct $i,j\in \{1,\dots ,N\}$, we have 
\[
v(x^{(i)},x^{(j)})\in \{k+1,\dots , k+N\}.
\]
Hence $d_{\alpha}(x^{(i)},x^{(j)})\ge \alpha(k)/2$. 
This implies that the set $\{x^{(1)},\dots,x^{(N)}\}$ is $(\alpha(k)/2)$-separated in $B(0,\alpha(k))$, 
and it has cardinality $N$.
Therefore, $(2^{\nn},d_{\alpha})$ is not doubling.
\end{proof}

On the uniform perfectness of $(2^{\nn},d_{\alpha})$, we also have the following:

\begin{lem}\label{lem:up}
Let $\alpha$ be a shrinking sequence. 
Then 
 $(2^{\nn},d_{\alpha})$ is uniformly perfect if and only if there exists $\rho\in (0,1)$ such that for all $n\in \nn$ 
we have
\begin{equation}\label{eq:upness}
\rho \alpha(n)\le \alpha(k)
\end{equation}
for some $k>n$.
\end{lem}
\begin{proof}
First we show that the condition \eqref{eq:upness} for some $\rho\in (0,1)$ implies the uniform perfectness. 
Take $x\in 2^{\nn}$ and $r\in (0,\di(X))$.
Choose $n\in \nn$ with $r\in [\alpha(n+1),\alpha(n))$.
Note that $B(x,r)=B(x,\alpha(n+1))$.
Since for some $k>n$ we have 
\[
\rho r<\rho \alpha(n)\le \alpha(k)\le \alpha(n+1),
\]
and since 
there exists $y\in B(x,r)$ with $v(x,y)=k$, 
we see that the set $B(x,r)\setminus B(x,\rho r)$ is non-empty.
Hence $(2^{\nn},d_{\alpha})$ is uniformly perfect.\par
Second, we show the contrary. 
Assume that for every $\rho\in (0,1)$ there exists $n\in \nn$ such that for every $k>n$ 
we have $\rho \alpha(n)>\alpha(k)$.
In this case, we can choose $m\in \nn$ satisfying $\alpha(m+1)<\alpha(m)$ and $\alpha(m+1)<\rho \alpha(m)$.
Take $r\in (\alpha(m+1)/\rho, \alpha(m))$, then
\[
B(0,r)=B(0,\rho r)=B(0,\alpha(m+1)).
\]
Therefore, the set $B(0,r)\setminus B(0,\rho r)$ is empty. 
This implies that $(2^{\nn},d_{\alpha})$ is not uniformly perfect.
\end{proof}
From Lemma \ref{lem:up} we can deduce the following characterization of the non-uniform perfectness:
\begin{lem}\label{lem:per}
Let $\alpha$ be a shrinking sequence. 
Then 
 $(2^{\nn},d_{\alpha})$ is not uniformly perfect if and only if there exists a function $\varphi:\nn\to \nn$ with
\begin{equation}\label{eq:pcdn}
\lim_{n\to \mgn}\frac{\alpha(\varphi(n)+1)}{\alpha(\varphi(n))}=0.
\end{equation}
\end{lem}

\subsection{Concrete Examples}
We next apply the previous lemmas to our construction of examples.\par
For $u\in (0,1)$, let $[u]$ denote the shrinking sequence defined by $[u](n)=u^n$. 
Then we have:
\begin{lem}\label{lem:111}
For every $u\in (0,1)$, the Cantor space $(2^{\nn}, d_{[u]})$ has type $(1,1,1)$.
\end{lem}
\begin{proof}
The shrinking sequence $[u]$ satisfies \eqref{eq:dbp} and \eqref{eq:upness}.  
Lemmas \ref{lem:db} and \ref{lem:up} imply that  $(2^{\nn}, d_{[u]})$ has type $(1,1,1)$.
\end{proof}

\begin{rmk}
The metric $d_{[1/3]}$ on $2^{\nn}$ coincides with the metric $e$ mentioned in Example \ref{exm:sym}. 
From the same argument as in the proof of Lemma \ref{lem:111}, 
we deduce that $(2^{\nn},e)$ has type $(1,1,1)$.
\end{rmk}

For a shrinking sequence $\alpha$, and for  $m\in \nn$,  
we define the $m$-shifted shrinking sequence $\alpha^{\{m\}}$ of $\alpha$ by $\alpha^{\{m\}}(n)=\alpha(n+m-1)$. 
Note that $(2^{\nn},d_{\alpha^{\{m\}}})$ is isometric to a closed ball $B(x,\alpha(m))$ in $(2^{\nn}, d_{\alpha})$. \par
By Lemmas \ref{lem:db} and \ref{lem:up}, we obtain the following two lemmas: 
\begin {lem}\label{lem:mdb}
Let $\alpha$ be a shrinking sequence. 
The space $(2^{\nn},d_{\alpha})$ is doubling if and only if for each $m\in \nn$ 
the space $(2^{\nn},d_{\alpha^{\{m\}}})$ is doubling. 
\end{lem}
\begin {lem}\label{lem:mup}
Let $\alpha$ be a shrinking sequence. 
The space $(2^{\nn},d_{\alpha})$ is uniformly perfect if and only if for each $m\in \nn$ 
the space $(2^{\nn},d_{\alpha^{\{m\}}})$ is uniformly perfect.  
\end{lem}
\begin{rmk}\label{rmk:sametype}
By Lemmas \ref{lem:mdb} and \ref{lem:mup}, and by the hereditary of the uniform disconnectedness, 
for every shrinking sequence $\alpha$, we see that 
every closed ball in $(2^{\nn},d_{\alpha})$ has the same type as $(2^{\nn}, d_{\alpha})$. 
\end{rmk}
We quest the types realized by sequentially metrized Cantor spaces.
\begin{lem}\label{lem:011}
Let $\alpha$ be a shrinking sequence defined by $\alpha(n)=1/n$. 
Then the Cantor space $(2^{\nn},d_{\alpha})$ has type $(0,1,1)$. 
\end{lem}
\begin{proof}
Since $\alpha(n)/2=\alpha(2n)$ for all $n\in \nn$, 
the sequence $\alpha$ satisfies \eqref{eq:upness} and does not satisfy \eqref{eq:dbp}.
By Lemmas \ref{lem:db} and \ref{lem:up}, the space $(2^{\nn},d_{\alpha})$ has type $(0,1,1)$. 
\end{proof}

\begin{lem}\label{lem:110}
Let $\beta$ be a shrinking sequence defined by $\beta(n)=1/n!$. 
Then the Cantor space $(2^{\nn},d_{\beta})$ has type $(1,1,0)$.
\end{lem}
\begin{proof}
For each $\rho\in (0,1)$, choose $m\in \nn$ with $1/m<\rho$. 
Then we have $\beta(n+1)\le \rho \beta(n)$ for all $n>m$.
Hence the sequence $\beta$ satisfies \eqref{eq:dbp} and does not satisfy \eqref{eq:upness}.
From Lemmas \ref{lem:db} and \ref{lem:up} it follows that the space $(2^{\nn},d_{\beta})$ has type $(1,1,0)$. 
\end{proof}

\begin{lem}\label{lem:010}
There exists a shrinking sequence $\gamma$ for which the Cantor space $(2^{\nn},d_{\gamma})$ has type $(0,1,0)$. 
\end{lem}
\begin{proof}
Let $\beta$ be the shrinking sequence defined by $\beta(n)=1/n!$.
For each $n\in \nn$, choose distinct $n$ numbers $r_{1,n},r_{2,n},\dots,r_{n,n}$ in the set 
\[
(\beta(2n-1)/2,\beta(2n-1)).
\]
Define the shrinking sequence $\gamma$ as the renumbering of 
\[
\beta(\nn)\cup\{\,r_{i,n}\mid n\in \nn,i\in\{1,\dots,n\}\,\}
\]
in decreasing order.
Since for each $n\in \nn$ the set
\[
\gamma(\nn)\cap (\beta(2n-1)/2,\beta(2n-1)))
\]
has cardinality $n$, the sequence $\gamma$ does not satisfy (\ref{eq:dbp}).
Define a function $\varphi:\nn\to \nn$ by $\varphi(n)=\gamma^{-1}(1/(2n-1)!)$. 
Then $\varphi$ satisfies 
\[
\gamma(\varphi(n))=1/(2n-1)!, \quad \gamma(\varphi(n)+1)=1/(2n)!.
\] 
From  Lemmas \ref{lem:db} and \ref{lem:per}, we deduce that $(2^{\nn},d_{\gamma})$ has type $(0,1,0)$.
\end{proof}

Using the sequentially metrized Cantor spaces, we see the following (cf. Lemma \ref{lem:pup}):

\begin{prop}\label{prop:nonup}
There exist shrinking sequences $\sigma$ and $\tau$ satisfying the following:
\begin{enumerate}
\item $(2^{\nn},d_{\sigma})$ and $(2^{\nn},d_{\tau})$ have type $(1,1,0)$;
\item $(2^{\nn}\times 2^{\nn},d_{\sigma}\times d_{\tau})$ is quasi-symmetrically equivalent to $(\Gamma,d_{\Gamma})$.
\end{enumerate} 
\end{prop}
\begin{proof}
(1) Let $\beta$ be the shrinking sequence defined by $\beta(n)=1/n!$.
Define a function $f:\nn\to \nn$ by 
\[
f(n)=\max\{\, k\in \nn\mid 2^{-k}\beta(n)\ge \beta(n+1)\,\}.
\]
We define the shrinking sequence $\sigma$ as the renumbering of the set 
\[
\beta(\nn)\cup \{\,2^{-i}\beta(2n)\mid n\in \nn, i=1,\dots ,f(2n)\,\}
\]
in decreasing order. 
We also define the shrinking sequence $\tau$ as the renumbering  of the set 
\[
\beta(\nn)\cup \{\,2^{-i}\beta(2n+1)\mid n\in \nn, i=1,\dots, f(2n+1)\,\}
\]
in decreasing order.\par
Define a function $\varphi:\nn\to \nn$ by $\varphi(n)=\sigma^{-1}(1/(2n-1)!)$ 
and $\psi :\nn\to \nn$ by $\psi(n)=\tau^{-1}(1/(2n)!)$. 
Note that $\varphi$ satisfies
\[
\sigma(\varphi(n))=1/(2n-1)!,\quad \sigma(\varphi(n)+1)=1/(2n)!
\]
and $\psi$ satisfies 
\[
\tau(\psi(n))=1/(2n)!, \quad \tau(\psi(n)+1)=1/(2n+1)!.
\]
Then $\varphi$ and $\psi$ satisfy \eqref{eq:pcdn}, and hence  by Lemma \ref{lem:per}, 
 both $(2^{\nn},d_{\sigma})$ and $(2^{\nn},d_{\tau})$ have type $(1,1,0)$. \par
(2) By Lemmas \ref{lem:pd} and \ref{lem:pud}, the space 
$(2^{\nn}\times 2^{\nn},d_{\sigma}\times d_{\tau})$ 
is doubling and uniformly disconnected. 
By the David-Semmes uniformization theorem (\cite[Proposition 15.11]{DS}),
it suffices to prove that $(2^{\nn}\times 2^{\nn},d_{\sigma}\times d_{\tau})$ is uniformly perfect.
Take $z=(x,y)\in 2^{\nn}\times 2^{\nn}$ and $r\in (0, \di(X\times Y))$. 
There exists $n\in \nn$ with  $r\in (\beta(n+1),\beta(n)]$.
If $n$ is even, 
then there exists $i\in \nn$ with $\sigma(i)\in (r/2, r)$. 
Hence the set $B(x,r)\setminus U(r/2)$ in $(2^{\nn}, d_{\sigma})$ is non-empty. 
Choose $x'\in B(x,r)\setminus U(r/2)$, and put $z'=(x',y)$. 
Since $(d_{\sigma}\times d_{\tau})(z,z')$ is equal to $d_{\sigma}(x,x')$, 
it belongs to $[r/2,r]$. 
Therefore, $B(z,r)\setminus U(z,r/2)$ in $(2^{\nn}\times 2^{\nn}, d_{\sigma}\times d_{\tau})$ is non-empty. 
If $n$ is odd, 
then there exists $j\in \nn$ with $\tau(j)\in (r/2,r)$. 
Hence the set $B(y,r)\setminus U(y,r/2)$ in $(2^{\nn},d_{\tau})$ is non-empty. 
Similarly to the case where $n$ is even, 
we see that the set $B(z,r)\setminus U(z,r/2)$ is non-empty. 
Thus $(2^{\nn}\times 2^{\nn}, d_{\sigma}\times d_{\tau})$ is $(1/2)$-uniformly perfect.
\end{proof}

\section{Totally Exotic Cantor Metric Spaces}\label{sec:totexo}
In this section, we prove Theorem \ref{thm:totexo}. 
In  Section \ref{sec:smcs}, 
we already know the existence of some totally exotic Cantor metric spaces for the doubling property and the uniformly perfectness. 
Using Lemmas \ref{lem:mdb} and \ref{lem:mup}, 
we obtain the following three propositions (see Remark \ref{rmk:sametype}): 
\begin{prop}\label{prop:011X}
Let $(X,d)$ be the Cantor metric space stated in Lemma \ref{lem:011}. 
Then $(X,d)$ has totally exotic type $(0,1,1)$. 
\end{prop}
\begin{prop}\label{prop:110X}
Let $(X,d)$ be the Cantor metric space stated in Lemma \ref{lem:110}. 
Then $(X,d)$ has totally exotic type $(1,1,0)$. 
\end{prop}
\begin{prop}\label{prop:010X}
Let $(X,d)$ be the Cantor metric space stated in Lemma \ref{lem:010}. 
Then $(X,d)$ has totally exotic type $(0,1,0)$. 
\end{prop}
For the proof of Theorem \ref{thm:totexo}, we construct totally exotic Cantor metric spaces for the uniform disconnectedness. 
Note that such spaces can not be constructed as sequentially metrized Cantor spaces.
\par
We introduce the notion of the kaleidoscope spaces. 
\begin{df}
For each $n\in \nn$, we define a subset $K_n$ of $\rr$ by
\[
K_n=\{\,k/n\mid k\in \{0,\dots,n\}\,\}, 
\]
and we denote by $d_n$ the metric of $K_n$ induced from $d_{\rr}$. 
Note that for each $n\in \nn$, the space $(K_n,d_n)$ has a $(1/n)$-chain, and it is $3$-doubling, 
 and that for each $x\in K_n$, we have $B(x,r)=\{x\}$ in $(K_n,d_n)$ if and only if $r<1/n$. 
Let $a:\nn\to (0,\infty)$ be a sequence satisfying 
\begin{equation}\label{n+1}
2na_n< a_{n+1}
\end{equation}
for all $n$, and 
\begin{equation}\label{kn}
ka_k<a_n 
\end{equation}
for all $n$ and $k<n$. 
Put $K(a)=\prod_{n\in \nn}K_n$, and define a metric $d_{K(a)}$ on $K(a)$ by
\[
d_{K(a)}(x,y)=\sup_{n\in \nn}\frac{1}{a_n}d_n(x_n,y_n),
\]
where $x=(x_n)$ and $y=(y_n)$.
We call  $(K(a),d_{K(a)})$ the \emph{kaleidoscope space of $a$}. 
Since the metric $d_{K(a)}$ on $K(a)$ induces the product topology of the family $\{K_n\}_{n\in \nn}$, 
the Brouwer theorem \ref{Brouwer} tells us that $(K(a),d_{K(a)})$ is a Cantor space. 
\end{df}
\begin{rmk}
By replacing the product factors in the construction of the kaleidoscope space of $a$ with $\{0,1\}$,  we obtain the sequentially metrized Cantor space metrized by $1/a$. 
\end{rmk}

\begin{lem}\label{lem:clbll}
Let $a:\nn\to (0,\infty)$ be a sequence  satisfying \eqref{n+1} and \eqref{kn}. 
Let $r\in (0,\mgn)$ and $x\in K(a)$. 
Take  $n\in \nn$ with $r\in [1/a_{n+1},1/a_{n})$. 
Then 
\begin{equation*}
B(x,r)=\{x_1\}\times\cdots \times\{x_{n-1}\}\times B(x_{n},ra_{n})\times \prod_{i>n}K_i.
\end{equation*}
\end{lem}
\begin{proof}
By the definition of $d_{K(a)}$, we have 
\[
B(x,r)=\prod_{i\in \nn}B(x_i,a_ir).
\] 
For every $y\in B(x,r)$, by (\ref{kn}), for all $k<n$ we have 
\[
d_k(x_k,y_k)\le ra_k<\frac{a_k}{a_{n}}<\frac{1}{k},
\]
and hence $x_k=y_k$.  
Therefore $B(x_k,a_kr)=\{x_k\}$ for all $k<n$.
For each $i>n$, by $a_{n+1}r\ge 1$ we have $a_ir\ge 1$. 
Hence $B(x_i,a_ir)=K_i$. 
Therefore we obtain the claim. 
\end{proof}
Similary to Lemma \ref{lem:clbll}, we can prove:
\begin{lem}\label{lem:opbll}
Let $a:\nn\to (0,\infty)$ be a sequence satisfying \eqref{n+1} and \eqref{kn}. 
Let $r\in (0,\mgn)$ and $x\in K(a)$. 
Take  $n\in \nn$ with $r\in [1/a_{n+1},1/a_{n})$. 
Then 
\begin{equation*}
U(x,r)=\{x_1\}\times\cdots \times\{x_{n-1}\}\times U(x_{n},ra_{n})\times \prod_{i>n}K_i.
\end{equation*}
\end{lem}

We next prove the doubling property of kaleidoscope spaces. 
\begin{lem}\label{lem:kld}
Let $a:\nn\to (0,\infty)$ be a sequence  satisfying \eqref{n+1} and \eqref{kn}. Then $(K(a),d_{K(a)})$ is $3$-doubling.
\end{lem}
\begin{proof}
Let $r\in (0,\infty)$, and take $n\in \nn$ with $r\in [1/a_{n+1}, 1/a_n)$. 
\par
Case (i): 
First we consider the case where $1/n\le a_{n}r$. 
We can take points  $p_1, p_2\in K_n$ such that 
\[
B(x_{n},ra_{n})\subset B(p_1,ra_{n}/2)\cup B(p_2,ra_{n}/2)\cup B(x_n,ra_{n}/2)
\] 
holds in $(K_{n},d_n)$. 
For each $j\in \{1,2\}$, define $q^{(j)}\in K(a)$ by 
\[
q^{(j)}_i=
 \begin{cases}
  x_i & \text{if $i\neq n$,}\\ 
  p_j  & \text{if $i=n$.} 
 \end{cases}
\]
By (\ref{n+1})  and the assumption $1/n\le a_{n}r$, 
for each $i>n$,  we have $ra_{i}/2>1$ 
and hence $B(x_{i},ra_{i}/2)=K_{i}$. 
Then
\[
B(q^{(j)},r/2)=\{x_1\}\times\cdots \times\{x_{n-1}\}\times B(p_j,a_{n}r/2)\times \prod_{i>n}K_i
\]
holds in $(K(a),d_{K(a)})$.
Therefore, so does
\[
B(x,r)\subset  B(q^{(1)},r/2)\cup B(q^{(2)},r/2)\cup B(x,r/2).
\]
Namely, $B(x,r)$ can be covered by at most $3$ balls with radius $r/2$.\par
Case (ii): Second we consider  the case where $a_{n}r<1/n$. 
In this case, $B(x_{n},a_{n}r)=\{x_n\}$. 
We can take points  $p_1, p_2\in K_{n+1}$ such that 
\[
B(x_{n+1},ra_{n+1})\subset B(p_1,ra_{n+1}/2)\cup B(p_2,ra_{n+1}/2)\cup B(x_{n+1},ra_{n+1}/2)
\] 
holds in $(K_{n+1},d_{n+1})$. 
Since for each $i>n+1$ we have $a_{i}r/2\ge 1$, 
\[
B(x,r/2)=\{x_1\}\times\cdots\times \{x_n\}\times B(x_{n+1},a_{n+1}r/2)\times \prod_{i>n+1}K_i.
\]
Hence, similary to Case (i), by defining $q^{(1)}, q^{(2)}\in K(a)$ appropriately, 
we can prove that 
$B(x,r)$ can be covered by at most $3$ balls with radius $r/2$.\par
Thus we conclude that $(K(a),d_{K(a)})$ is $3$-doubling.
\end{proof}

Since for each $n\in \nn$ the space $(K_n,d_n)$ has a $(1/n)$-chain,  we see:
\begin{lem}\label{lem:klud}
Let $a:\nn\to (0,\infty)$ be a sequence  satisfying \eqref{n+1} and \eqref{kn}. 
Then we have
\[
S_{UD}(K(a),d_{K(a)})=K(a). 
\]
\end{lem}

The idea of kaleidoscope spaces provides us examples of totally exotic Cantor metric spaces of remaining types. 
\begin{prop}\label{prop:101X}
There exists a  Cantor metric space  of totally exotic type $(1,0,1)$.
\end{prop}
\begin{proof}
Define a sequence  $a:\nn\to (0,\infty)$ by $a_n=2^n\cdot n!$. Then the sequence $a$ satisfies \eqref{n+1} and \eqref{kn}. 
By Lemmas \ref{lem:kld} and \ref{lem:klud}, we see that $S_{UD}(K(a),d_{K(a)})=K(a)$, and that $(K(a),d_{K(a)})$ is doubling and non-uniformly disconnected.
\par 
We are going to prove that $(K(a),d_{K(a)})$ is $(1/16)$-uniformly perfect. 
To do this, for each $x\in E$ and for each $r\in (0,1/2)$, we show that the set $B(x,r)\setminus U(x,r/16)$ is non-empty. Take $n\in \nn$ with $r\in [1/a_{n+1}, 1/a_n)$. 
\par
Case (i): Assume $B(x_n,ra_{n})= K_{n}$. 
If $n=1$, then  $ra_1\le 1/4$; 
if $n>1$, then  $ra_{n-1}\le 1/(n-1)$ and hence  $ra_{n}\le 2n/(n-1)$. 
In any case, we have $ra_n\le 4$. 
Then we obtain 
\[
\di(U(x_n,ra_n/16))\le 1/2.
\]
By $\di(K_n)=1$, we see that $B(x_n,ra_n)\setminus U(x_n,ra_n/16)$ is non-empty. 
Hence so is $B(x,r)\setminus U(x,r/16)$.
\par
Case (ii): Assume $B(x_n,ra_{n})\neq K_{n}$ and $B(x_n,ra_n)\neq \{x_n\}$.
Take an end point $y\in K_n$ of $B(x_n,ra_n)$. 
Without loss of generality, 
by considering the map defined by $t\mapsto -t+1$,  
we may assume that $y$ is the right end point of $B(x_n,ra_n)$ and $y\neq 1$. 
By the assumption $B(x_n,ra_n)\neq \{x_n\}$, we may also assume $y\neq x_n$. 
Note that $y$ is the maximum of  $B(x_n,ra_n)$.
Define a point $z\in K(a)$ by 
\[
z=\begin{cases}
		x_i & \text{if $i\neq n$,}\\
		y & \text{if $i=n$.}
	\end{cases}
\]
Then we have $d_{K(a)}(x,z)\le r$. 
By the construction of $K_n$, 
we may assume that $y=x+m/n$ holds for some positive integer $m\le n$ with 
\[
\frac{m}{n}\le ra_n< \frac{m+1}{n}.
\]
This implies 
\[
d_{K(a)}(x,z)=\frac{1}{a_n}\frac{m}{n}\ge \frac{1}{a_n}\frac{1}{16}\frac{m+1}{n}>\frac{1}{16}r.
\]
 Hence $B(x,r)\setminus U(x,r/16)$ is non-empty. \par
Case (iii): Assume $B(x_n,ra_{n})= \{x_{n}\}$. Then $ra_n<1/n$, and 
\[
\di(U(x_{n+1}, ra_{n+1}/16))\le \frac{ra_{n+1}}{8}=\frac{ra_{n}\cdot 2(n+1)}{8}<\frac{1}{2}.
\]
Hence $U(x_{n+1}, ra_{n+1}/16)\neq K_{n+1}$. 
Recall that  $B(x_{n+1}, ra_{n+1})=K_{n+1}$. 
Therefore the set $B(x,r)\setminus U(x,r/16)$ is non-empty. 
Thus we conclude that $(K(a),d_{K(a)})$ is a desired space. 
\end{proof}

\begin{rmk} 
It is known that any subset of $\rr$ with positive Lebesgue measure is not uniformly disconnected 
(see e.g., \cite[Corollary 4.6]{L}).
Let $A$ be a Cantor space in $\rr$ whose every non-empty open subset has positive Lebesgue measure.
By the arguments in Subsection  \ref{sub:prod}, 
we see that $A\times \Gamma$ also has type $(1,0,1)$ and 
that non-empty open set of $A\times \Gamma$ is not uniformly disconnected. 
The author does not know whether such $A$ is uniformly perfect or not.
\end{rmk}


\begin{prop}\label{prop:100X}
There exists a Cantor metric space of totally exotic type $(1,0,0)$. 
\end{prop}
\begin{proof}
Define a sequence  $b:\nn\to (0,\infty)$ by $b_n=(2n)!$. 
Then the sequence  $b$ satisfies \eqref{n+1} and \eqref{kn}. 
We prove that $(K(b),d_{K(b)})$ is a desired space. 
By Lemmas \ref{lem:kld} and \ref{lem:klud}, 
we see that 
the space $(K(b),d_{K(b)})$ is doubling and satisfies $S_{UD}(K(b),d_{K(b)})=K(b)$. 
We show that $S_{UP}(K(b),d_{K(b)})=K(b)$. 
Namely, 
we show that for each $x\in K(b)$, and  for each $\rho\in (0,1]$, 
there exists $r\in (0,\di(K(b)))$ such that 
$B(x,r)\setminus U(x,\rho r)=\emptyset$.  
For each $\rho \in (0,1]$, we can take $n\in \nn$ with $\rho b_{n+1}/2nb_n>1$. 
Let $r=(2nb_n)^{-1}$. Then $r\in [1/b_{n+1},1/b_n)$. 
Since $b_nr<1/n$, we have
\[
B(x,r)=\{x_1\}\times\cdots \times\{x_{n-1}\}\times \{x_n\}\times \prod_{i>n}K_i.
\]
From $\rho rb_{n+1}>1$ we derive
\[
U(x,\rho r)=\{x_1\}\times\cdots \times\{x_{n-1}\}\times \{x_n\}\times \prod_{i>n}K_i.
\]
Therefore $B(x,r)=U(x,\rho r)$, hence 
 $S_{UP}(K(b),d_{K(b)})=K(b)$. 
\end{proof}
By modifying the product factors in the construction of the kaleidoscope spaces, we obtain:
\begin{prop}\label{prop:000X}
There exists a Cantor metric space  of totally exotic type $(0,0,0)$. 
\end{prop}
\begin{proof}
For each $n\in \nn$, take an $(n, 1/2n)$-discrete space $(A_n, e_n)$ (see Definition \ref{def:discrete}). 
Put $(L_n,D_n)=(A_n\times E_n, e_n\times d_n)$.  Let
\[
L=\prod_{i\in \nn}L_i, 
\]
and define the metric $d_L$ on $L$ by  
\[
d_L(x,y)=\sup_{n\in \nn}\frac{1}{b_n}D_n(x_n,y_n), 
\]
where $b_n=(2n)!$.
We prove that  $(L,d_L)$ is a desired space. 
Since $B(x,r)$ has  $(1/n)$-chains and  $n$-discrete subspaces for all sufficiently large $n$, 
 we have  $S_{D}(L,d_L)=L$ and $S_{UD}(L,d_L)=L$. 
We next show $S_{UP}(L,d_L)=L$. 
Let $x\in L$ and $\rho \in (0,1]$. 
We can take $n\in \nn$ with $\rho b_{n+1}/4nb_n>1$. Put $r=(4nb_n)^{-1}$. 
Since $b_n r<1/2n$, 
similary to Lemma \ref{lem:clbll}, we see
\[
B(x,r)=\{x_1\}\times \cdots \{x_n\}\times \prod_{i>n}L_i. 
\]
Since $\rho b_{n+1}/4nb_n>1$, we have 
\[
U(x,\rho r)=\{x_1\}\times \cdots \{x_n\}\times \prod_{i>n}L_i. 
\]
Hence $B(x,r)\setminus U(x,\rho r)$ is empty. 
Therefore $S_{UP}(L,d_L)=L$. 
\end{proof}
To finish the proof of Theorem \ref{thm:totexo}, we next show the following:
\begin{prop}\label{prop:001X}
There exists a Cantor metric space of totally exotic type $(0,0,1)$. 
\end{prop}
\begin{proof}
By Propositions \ref{prop:011X} and \ref{prop:101X}, we can take 
a Cantor metric space $(X,d_X)$ of totally exotic type $(0,1,1)$, and 
a Cantor metric space $(Y,d_Y)$ of totally exotic type $(1,0,1)$. 
Using Lemmas \ref{lem:pd}, \ref{lem:pud} and \ref{lem:pup}, 
we see that the space $(X\times Y,d_X\times d_Y)$ has totally exotic type $(0,0,1)$.  
\end{proof}

\begin{proof}[Proof of Theorem \ref{thm:totexo}]
Propositions \ref{prop:011X}--\ref{prop:001X} complete the proof. 
\end{proof}

\section{Prescribed Hausdorff and Assouad Dimensions}\label{sec:prescribe}
In this section, we prove Theorem \ref{thm:ab}.
\subsection{Basics of Assouad Dimension}\label{Assouad}
Let $(X,d)$ be a metric space. Define a function $\mathcal{N}:(0,2)\to \nn\cup \{\infty\}$ by 
assigning $\mathcal{N}(\epsilon)$ to the infimum of $N\in \nn$ such that every closed metric ball in $(X,d)$ with radius 
$r$ can be covered by at most $N$ closed metric balls with radius $\epsilon r$. 
The Assouad dimension $\dim_A (X,d)$ of $(X,d)$ is defined as 
the infimum of $s\in (0,\infty)$ for which there exists $K\in (0,\infty)$ such that 
for all $\epsilon\in (0,2)$ we have 
\[
\mathcal{N}(\epsilon)\le K\epsilon^{-s}.
\]
Note that  $(X,d)$ is doubling if and only if  $\dim_A(X,d)$ is finite.\par
Let $A$ be a subset of $X$. 
Define a function $\mathcal{M}:(0,2)\to \nn\cup\{\infty\}$ by 
assigning $\mathcal{M}(\epsilon)$ to the supremum of the cardinality of $(\epsilon r)$-separated sets of closed metric balls with radius $r$. 
Note that  for every $\epsilon\in (0,2)$ we have
\[
\mathcal{M}(3\epsilon)\le \mathcal{N}(\epsilon)\le \mathcal{M}(\epsilon).
\]
Moreover, $\dim_A(X,d)$ is equal to the infimum of $s\in (0,\infty)$ for which 
there exists $K\in (0,\infty)$ such that for all $\epsilon\in (0,2)$ 
we have 
\[
\mathcal{M}(\epsilon)\le K\epsilon^{-s}.
\]
The Assouad dimension satisfies the following finite stability:
\begin{prop}
Let $A$ and $B$ be subsets of a metric space. Then 
\[
\dim_A(A\cup B)=\max\{\dim_A(A),\dim_A(B)\}.
\]
\end{prop}
The Assouad dimension can be estimated from above as follows:
\begin{lem}\label{lem:es}
Let $\lambda\in (0,1)$. Let $(X,d)$ be a metric space. If every closed ball in $(X,d)$ with radius $r$ can be covered by at most $N$ closed balls with radius $\lambda r$, then we have
\[
\dim_A(X,d)\le \frac{\log (N)}{\log (\lambda^{-1})}.
\]
\end{lem}
For a positive number $\epsilon\in (0,\infty)$, and for a metric space $(X,d)$, 
the function $d^{\epsilon}$ is said to be a \emph{snowflake of $d$ with parameter $\epsilon$}		 
if $d^{\epsilon}$ is a metric on $X$. 
Note that the induced topology from $d^{\epsilon}$ coincides with the original one.

\begin{rmk}
Let $(X,d)$ be a metric space.
If $\epsilon\in (0,1)$, then $d^{\epsilon}$ is a metric on $X$. 
If $d$ is an ultrametric on $X$, 
then so is $d^{\epsilon}$ for any $\epsilon\in (0,\infty)$. 
\end{rmk}
For the snowflakes, we have:
\begin{lem}\label{lem:hol}
Let $\epsilon \in (0,\infty)$. Let $(X,d)$ be a metric space. 
If $d^{\epsilon}$ is a snowflake of $d$ with parameter $\epsilon$, then we have 
\[
\dim_A(X,d^{\epsilon})=\frac{1}{\epsilon}\dim_A(X,d).
\]
\end{lem}
From the definitions, we see the following:
\begin{prop}\label{prop:ha}
The Hausdorff dimension does not succeed the Assouad dimension.
\end{prop}

\subsection{Prescribed Dimensions}
We first calculate the Assoud dimension of the Cantor metric space mentioned in Lemma \ref{lem:110}. 
\begin{lem}\label{lem:00}
Let $\beta$ be a shrinking sequence defined by $\beta(n)=1/n!$. 
Then 
\[
\dim_A(2^{\nn},d_{\beta})=0.
\]
In particular, $\dim_H(2^{\nn},d_{\beta})=0$.
\end{lem}
\begin{proof}
For each $k\in \nn$, let $n(k)\in \nn$ be  the integer satisfying 
\begin{align}
\label{eq:k}
\frac{1}{(n(k)+1)!}< \frac{1}{k} \le\frac{1}{n(k)!}.
\end{align}
For a fixed $r\in (0,\infty)$, let $m\in \nn$ be the least positive integer with
\[
\frac{1}{(m+1)!}\le r.
\]
Since $B(x,r)$ coincides with $B(x,1/(m+1)!)$, we have
\[
B(x,r)=\{y\in 2^{\nn}\mid v(x,y)\ge m\}.
\]
Let $T_k$ be the subset of $B(x,r)$ consisting of all points $y\in B(x,r)$ such that $y_i=0$ for all $i> m+n(k).$
Then $\card (T_k)= 2^{n(k)+1}$. 
For every $y\in B(x,r)$, there exists  $z\in T_k$ such that $v(y,z)\ge m+n(k)+1$,  
and hence we have
\begin{align*}
d_{\beta}(y,z)&\le \frac{1}{(m+n(k)+1)!}\le \frac{1}{(m+1)!}\frac{1}{(m+2)\cdots (m+n(k)+1)}\\
&\le \frac{1}{(m+1)!}\frac{1}{(n(k)+1)!}< \frac{r}{k}.
\end{align*}
Therefore every closed ball in $(2^{\nn},d_{\beta})$ with radius $r$ can be covered by at most $2^{n(k)+1}$ balls with radius $r/k$. 
By Lemma \ref{lem:es}, we have 
\begin{align*}
\dim_A(2^{\nn},d_{\beta})\le \frac{n(k)+1}{\log k}.
\end{align*}
Using \eqref{eq:k}, we estimate
\[
\frac{n(k)+1}{\log k}\le \frac{n(k)+1}{\log1+\log2+\cdots +\log n(k)}.
\]
The right hand side tends to $0$ as $k\to \infty$. 
Hence $\dim_A(2^{\nn},d_{\beta})=0$; in particular, 
by Proposition \ref{prop:ha} 
we have $\dim_H(2^{\nn},d_{\beta})=0$. 
\end{proof}
The following sequentially metrized Cantor space plays a key role in the proof of Theorem \ref{thm:ab}. 
\begin{prop}\label{prop:theta}
There exists a shrinking sequence $\theta $ with 
\[
\dim_H(2^{\nn},d_{\theta})=0,\quad \dim_A(2^{\nn},d_{\theta})=1.
\]
\end{prop}
\begin{proof}
Take a shrinking sequence $\alpha$ defined by $\alpha (n)=2^{-n^{3}}$. Define a shrinking sequence $\theta$ by 
the renumbering of the set 
\[
\alpha(\nn)\cup \{\,2^{-k}\alpha(n)\mid n\in \nn, k=1,\dots, n\,\}
\] 
in decreasing order. 
Define a function $\varphi :\nn\to \nn$ by $\varphi(n)=n(n+1)/2$. 
Then $\theta(\varphi(n))=\alpha(n)=2^{-n^3}$  
and $\varphi(n)\le n^2$ hold for each $n\in \nn$.  \par
First we estimate the Hausdorff dimension.
For each finite sequence $\{i_k\}_{k=1}^m$ valued in $\{0,1\}$, we define
\[
S_{i_1,i_2,\dots,i_m}=\{\,x\in 2^{\nn}\mid x_1=i_1, x_2=i_2,\dots, x_m=i_m\,\}.
\]
Then for each fixed $m\in \nn$ we have 
\[
2^{\nn}=\bigcup_{i_1,i_2\dots,i_m}S_{i_1,i_2,\dots,i_m}.
\]
By $\di(S_{i_1,i_2,\dots,i_m})=\theta(m+1)$, for each $s\in (0,\mgn)$
\[
\mathcal{H}_{\theta(m+1)}^{s}(2^{\nn},d_{\theta})=\sum_{i_1,i_2,\dots,i_m}\di(S_{i_1,i_2,\dots,i_m})^s
=2^m\cdot (\theta(m+1))^s.
\]
Put $m=\varphi(n)-1$, then for each $s\in (0,\mgn)$, we see that 
\[
\mathcal{H}_{\alpha(n)}^{s}(2^{\nn},d_{\theta})= 2^{\varphi(n)-1}({2^{-n^3}})^s\le 2^{-sn^3+n^2-1}.
\]
Since $\alpha(n)$ and $ 2^{-sn^3+n^2-1}$ tend to $0$ as $n\to \infty$, 
we have $\mathcal{H}^s(2^{\nn},d_{\theta})=0$ for any $s\in (0,\infty)$. 
Hence $\dim_H(2^{\nn},d_{\theta})=0$.\par
Next, we prove $\dim_A(2^{\nn},d_{\theta})=1$.
Since $(2^{\nn},d_{\theta})$ is $2$-doubling, 
Lemma \ref{lem:es} implies $\dim_A(2^{\nn},d_{\theta})\le1$.
Take a number $t$ larger than $\dim_A(2^{\nn},d_{\theta})$ for which 
there exists $K\in (0,\mgn)$ such that for each $\epsilon \in (0,2)$ we have 
\begin{equation}\label{inq}
\mathcal{M}(\epsilon)\le K\epsilon^{-t}, 
\end{equation}
where $\mathcal{M}$ is the function defined in Subsection \ref{Assouad}.
For each $n\in \nn$, the ball $B(0,\alpha(n))$ in $(2^{\nn},d_{\theta})$ coincides with the set 
\[
\{\,y\in 2^{\nn}\mid v(x,y)\ge \varphi(n) \,\}.
\]
Let $T_n$ be the set of all points $z\in B(0,\alpha(n))$ such that $z_i=0$ for all $i>\varphi(n)+n$. 
We see that $T_n$ is an $(\alpha(n)/2^n)$-separated set in $B(0,\alpha(n))$ consisting of $2^{n+1}$ elements. 
Hence by (\ref{inq}) we have 
\[
2^{n+1}\le K2^{tn}.
\]
Since $K$ does not depend on $n$, we obtain $t\ge 1$. 
Then $\dim_A(2^{\nn},d_{\theta})\ge 1$. 
Therefore  $\dim_A(2^{\nn},d_{\theta})=1$.
\end{proof}
We next show the following:
\begin{lem}\label{lem:stn}
Take $u\in (0,1)$. Let $[u]$ be the shrinking sequence defined by $[u](n)=u^n$. 
Then we have 
\[
\dim_H(2^{\nn},d_{[u]})=\dim_A(2^{\nn},d_{[u]})=\frac{\log 2}{\log (u^{-1})}.
\]
\end{lem}
\begin{proof}
It is already known that 
\[
\dim_{H}(\Gamma,d_{\Gamma})=\frac{\log 2}{\log 3}.
\]
Then $(2^{\nn},d_{[1/3]})$ has the same Hausdorff dimension (see Example \ref{exm:t}). 
Put 
\[
c=\frac{\log (u^{-1})}{\log 3}.
\]
Since $(2^{\nn},d_{[1/3]}^c)$ coincides with $(2^{\nn},d_{[u]})$, we have 
\[
\dim_H(2^{\nn}.d_{[u]})=\frac{1}{c}\frac{\log 2}{\log 3}=\frac{\log 2}{\log (u^{-1})}.
\]
Next we estimate the Assouad dimension. 
Every closed ball in $(2^{\nn},d_{[u]})$ with radius $r$ can be covered by at most $2$ 
closed balls with radius $u r$. 
Then by Lemma \ref{lem:es}, we have 
\[
\dim_A(2^{\nn},d_{[u]})\le\frac{\log 2}{\log (u^{-1})}.
\]
Proposition \ref{prop:ha} completes the proof.
\end{proof}

We are going to prove Theorem \ref{thm:ab}. 
\begin{proof}[Proof of Theorem \ref{thm:ab}] We divide the proof into the following five cases.\par
Case (i): Assume $a=b=0$. The space $(2^{\nn},d_{\beta})$ mentioned in Lemma \ref{lem:00} satisfies the desired properties. \par 
Case (ii): Assume that  $a=0$ and $0<b<\infty$.  
Let $(2^{\nn},d_{\theta})$ be the space mentioned in Proposition \ref{prop:theta}. By Lemma \ref{lem:hol}, we have 
\[
\dim_H(2^{\nn},d_{\theta}^{1/b})=0,\quad \dim_A(2^{\nn},d_{\theta}^{1/b})=b.
\]\par
Case (iii): Assume that $0<a<\infty$ and $0<b<\infty$. 
Let $u=2^{-1/a}$. By Lemma \ref{lem:stn}, we have
\[
\dim_H(2^{\nn},d_{[u]})=\dim_A(2^{\nn},d_{[u]})=a.
\]
By the finite stabilities of the Hausdorff and Assouad dimensions, 
we see that the space $(2^{\nn}\sqcup 2^{\nn},  d_{[u]}\sqcup d_{\theta}^{1/b})$ satisfies the desired properties.\par 
Case (iv): Assume that  $0\le a<\infty$ and $b=\infty$.
We can take a Cantor space $(C,d)$ with $\dim_H(C,d)=a$ and $\di (C,d)=1/2$. 
For each $n\in \nn$, take disjoint $n$ copies $C_1, \dots C_n$ of $C$, 
and define a set $A_n$ by 
\[
A_n=\coprod_{i=1}^nC_i,
\]
and define a metric $e_n$ on $A_n$ by
\[
e_n(x,y)=
	\begin{cases}
	d(x,y) & \text{if $x,y\in C_i$ for some $i$,}\\
	1 & \text{otherwise.}
	\end{cases}
\]
Note that for each $n\in \nn$, the space $(A_n,e_n)$ is a Cantor space.
Let $\mathcal{A}=\{(A_i,2^{-n-1}d_i)\}_{i\in \nn}$.
For the telescope base $\mathcal{R}$ defined in Definition \ref{def:tbr}, 
the pair  $\mathcal{P}=(\mathcal{A},\mathcal{R})$ is compatible. 
By Lemma \ref{lem:tlCantor}, 
the telescope space $(T(\mathcal{P}),d_{\mathcal{P}})$ is a Cantor space.
By the countable stability of the Hausdorff dimension, 
we have $\dim_H(T(\mathcal{P}),d_{\mathcal{P}})=a$. 
Since for each $n\in \nn$ the space $(A_n,e_n)$ has an $n$-discrete subspace, 
by Lemma \ref{lem:dsc} the space $(T(\mathcal{P}),d_{\mathcal{P}})$ is not doubling. 
Namely, $\dim_A(T(\mathcal{P}),d_{\mathcal{P}})=\infty$.\par
Case (v): Assume $a=b=\infty$. 
For each $n\in \nn$, we can take a Cantor metric space $(T_n,d_n)$ with $\dim_H(T_n,d_n)=n$ 
and $\di (T_n,d_n)=2^{-n-1}$. Let $\mathcal{T}=\{(T_i,d_i)\}_{i\in\nn}$. 
For the telescope base $\mathcal{R}$ defined in Definition \ref{def:tbr}, 
the pair $\mathcal{Q}=(\mathcal{T},\mathcal{R})$ is compatible. 
By Lemma \ref{lem:tlCantor}, 
the telescope space $(T(\mathcal{Q}),d_{\mathcal{Q}})$ is a Cantor space.
By the countable stability of the Hausdorff dimension, 
we have $\dim_H(T(\mathcal{Q}),d_{\mathcal{Q}})=\infty$; 
in particular, Proposition \ref{prop:ha} implies $\dim_A(T(\mathcal{Q}),d_{\mathcal{Q}})=\infty$.\par
We have completed the proof of Theorem \ref{thm:ab}.
\end{proof}
\begin{rmk}
Let $\alpha$ be a shrinking sequence defined by $\alpha(n)=1/n$.
The Cantor space $(2^{\nn},d_{\alpha})$ mentioned in Proposition \ref{lem:011} also satisfies $\dim_H(2^{\nn},d_{\alpha})=\dim_A(2^{\nn},d_{\alpha})=\infty$. 
\end{rmk}

\end{document}